\documentclass[a4paper]{article}
\usepackage{geometry}
 \usepackage{amsmath,amsthm,amssymb,amsfonts, fancyhdr, xcolor, comment, graphicx, environ,mathtools,bm,bbm}
 
\usepackage{caption}
\usepackage{subcaption}
\usepackage[english]{babel}
\usepackage[utf8]{inputenc}
\usepackage{csquotes}
\usepackage{multirow}
\usepackage{dsfont}
\usepackage{authblk}
\newcommand{\1}{\mathds{1}}

\usepackage{hyperref}

\usepackage[round,sort]{natbib}

\newcommand{\BN}{}
\newcommand{\EN}{}
\DeclareMathOperator* {\argmin}{arg\,min}

\newcommand{\Rset}{{\mathbb R}}
\newtheorem{assumption}{Assumption}
\newtheorem{definition}{Definition}
\newtheorem{theorem}{Theorem}[section]
\newtheorem{corollary}{Corollary}[theorem]
\newtheorem{lemma}[theorem]{Lemma}
\newtheorem{proposition}[theorem]{Proposition}
\newtheorem{remark}{Remark}

\newcommand{\bx}{\bm{x}}
\newcommand{\bb}{\bm{b}}
\newcommand{\bbn}{\bm{b}_n}
\newcommand{\bA}{\bm{A}}
\newcommand{\bz}{\bm{0}}
\newcommand{\bc}{\bm{c}}
\newcommand{\bd}{\bm{d}}
\newcommand{\bv}{\bm{v}}
\newcommand{\be}{\bm{e}}
\newcommand{\RR}{\mathbb{R}}

\newcommand{\cN}{\mathcal{N}}
\DeclareMathOperator{\dom}{dom}
\DeclareMathOperator{\inter}{int}
\DeclareMathOperator{\relint}{relint}
\author[1]{Shuyu Liu}
\author[2]{Florentina Bunea}
\author[1,3]{Jonathan Niles-Weed}

\affil[1]{Courant Institute of Mathematical Sciences, New York University, New York, NY, 10012}
\affil[2]{Department of Statistics and Data Science, Cornell University, Ithaca, NY, 14853}
\affil[3]{Center for Data Science, New York University, New York, NY, 10012}

\title{Beyond entropic regularization:   Debiased Gaussian estimators for discrete optimal transport and general linear programs
}
\date{}

\begin{document}

\maketitle
\begin{abstract}
    This work proposes new estimators for discrete optimal transport plans that enjoy Gaussian limits centered at the true solution. This behavior stands in stark contrast with the performance of existing estimators, including those
based on entropic regularization, which are asymptotically biased and only
satisfy a CLT centered at a regularized version of the population-level plan.
We develop a new regularization approach based on a different class of
penalty functions, which can be viewed as the duals of those previously considered in the literature. The key feature of these penalty schemes it that they
give rise to preliminary estimates that are asymptotically linear in the penalization strength. Our final estimator is obtained by constructing an appropriate linear combination of two penalized solutions corresponding to two
different tuning parameters so that the bias introduced by the penalization
cancels out. Unlike classical debiasing procedures, therefore, our proposal entirely avoids the delicate problem of estimating and then subtracting the
estimated bias term. Our proofs, which apply beyond the case of optimal transport, are based on a novel asymptotic analysis of penalization schemes
for linear programs. As a corollary of our results, we obtain the consistency
of the naive bootstrap for fully data-driven inference on the true optimal solution. Simulation results and two data analyses support strongly the benefits of our approach relative to existing techniques.
\end{abstract}

\section{Introduction}\label{sec:intro}
Optimal transport is an increasingly prominent tool in data analysis, with applications in causal inference~\citep{charpentier2023optimal,wang2023optimal,torous2024optimal}, hypothesis testing~\citep{ghosal2022multivariate,deb2023multivariate,shi2022distribution}, and domain adaptation~\citep{redko2019optimal,rakotomamonjy2022optimal,courty2016optimal}, as well as in computational biology~\citep{schiebinger2019optimal,tameling2021colocalization,bunne2023learning,
huizing2022optimal,
tang2025optimal} and high-energy physics~\citep{komiske2019metric,
cai2020linearized,
park2023neural}.
A chief virtue of optimal transport---as compared with other measures of discrepancy between probability measures, such as classical statistical divergences~\citep{bhattacharyya1943measure,kullback1951information,ali1966general}, MMD~\citep{gretton2006kernel}, or integral probability metrics~\citep{muller1997integral,sriperumbudur2009integral}---is that the definition of optimal transport gives rise both to a measure of similarity between two probability measures but also an optimal \emph{coupling} or \textit{plan}, which describes the least costly way to move mass from one measure to another.
In fact, in many of the applications described above, it is the plan rather than the distance itself which is the primary object of interest~\citep[see, e.g.,][]{charpentier2023optimal,courty2016optimal,ghosal2022multivariate,schiebinger2019optimal}.
The goal of this paper is to define a new estimator of the plan, with better properties than existing approaches.

We focus on the discrete version of the optimal transport problem, which compares two distributions $\bm{t}$ and $\bm{s}$ supported on a finite set of points $\{v_1, \dots, v_p\}$.
Associated to each pair of points $\{v_i, v_j\}$ is a \emph{cost} $c_{ij}$.
The optimal transport problem then reads
\begin{equation}\label{eq:basic_ot_problem}
	\min_{\pi \in \RR^{p \times p}} \sum_{i,j=1}^p \pi_{ij}c_{ij} \quad \quad \text{s.t. } \pi \bm 1 = \bm t, \pi^\top \bm 1 = \bm s, \pi \geq \bm 0\,.
\end{equation}
A solution $\pi^\star$ is the \emph{optimal plan}, which is a joint probability measure on $\{v_1, \dots, v_p\}^{2}$.
The discrete optimal transport problem is the one most commonly used in applications, since from a practical perspective most optimal transport problems, including those on continuous domains, are solved by binning the data and solving a discrete optimal transport problem on a finite set.

In statistical contexts, the population-level marginal distributions $\bm t$ and $\bm s$ are not know exactly, and the statistician only has access to $n$ i.i.d.\ samples from each measure.
A natural ``plug-in'' estimator $\pi_n$ of the optimal plan is obtained by solving an empirical version of~\eqref{eq:basic_ot_problem}, with $\bm t$ and $\bm s$ replaced by empirical frequencies $\bm t_n$ and $\bm s_n$, where
\begin{equation}\label{eq:model}
	\begin{aligned}
		n \bm t_n & \sim \mathrm{Mult}(n, \bm t) \\
		n \bm s_n & \sim \mathrm{Mult}(n, \bm s).
	\end{aligned}
\end{equation}

Note that $\bm t_n$ and $\bm s_n$ are maximum-likelihood estimators of $(\bm t, \bm s)$, and, by extension,  the plug-in estimator $\pi_n$ is also the maximum-likelihood estimator of $\pi^\star$. 
More explicitly, we can view $\pi^\star = \pi^\star(\bm t, \bm s)$ as a functional of the marginal measures $\bm t$ and $\bm s$.
The performance of the plug-in estimator $\pi_n$ depends on the properties of the mapping $(\bm t, \bm s) \mapsto \pi^\star$.
As we show below, the linear programming structure of~\eqref{eq:basic_ot_problem} implies that this functional is \textit{non-smooth}.
If is well known that such non-smooth functional estimation problems possess particular challenges.~\citep{AitSil58,MOMBound,Che54} 
In our setting, this can be easily seen through the lens of asymptotic bias, already visible in the simplest optimal transport problems.

A small example will make this phenomenon clear.
Consider a $2 \times 2$ optimal transport problem, with cost $c_{ij} = \1\{i \neq j\}$, which leads to the program
\begin{equation}\label{equ: 2by2 example}
	\min_{\pi \in \mathbb{R}^{2 \times 2}} \pi_{12} + \pi_{21}\,, \qquad  \textrm{s.t.}\  {\pi} \bm{1} = \bm{t}, {\pi}^\top \bm{1} = \bm{s}, \pi \geq \bm{0}.
\end{equation}
When $\bm{t} = \bm{s} = (1/2, 1/2)$, then the unique optimal plan is given by $\pi^\star = (1/2, 0; 0, 1/2)$.
Now, consider empirical frequencies $\bm t_n = (t_{n, 1}, t_{n, 2})$ and $\bm s_n = (s_{n, 1}, s_{n, 2})$ as in~\eqref{eq:model}, and construct the corresponding plug-in estimator $\pi_n$.
The following result characterizes the asymptotic behavior of $\pi_n$.
\begin{proposition}\label{prop:2by2solution}
	The plug-in estimator obtained by solving~\eqref{equ: 2by2 example} with the empirical marginals $(\bm t_n, \bm s_n)$ is given by
	\begin{equation}\label{eq:pi_n expression}
		\pi_n = \begin{pmatrix}
			\min\{t_{n, 1}, s_{n, 1}\} & (t_{n, 1} - s_{n, 1})_+\\
			(t_{n, 2} - s_{n, 2})_+ & \min\{t_{n, 2}, s_{n, 2}\}\,
		\end{pmatrix}\,,
	\end{equation}
	
	As a consequence,
	\begin{equation}\label{eq:pi_n limit}
		\sqrt n( \pi_n-\pi^\star) \overset{d}{\to} \frac 12 \begin{pmatrix}
			\min\{g_1, g_2\} & (g_1 - g_2)_+ \\
			(g_2 - g_1)_+ & -\max\{g_1, g_2\}
		\end{pmatrix}\,,
	\end{equation}
	where $g_1$ and $g_2$ are independent standard Gaussian random variables.
\end{proposition}
Examining~\eqref{eq:pi_n expression}, we see that the plug-in estimator $\pi_n$ is a non-smooth function of the empirical frequencies~$(\bm t_n, \bm s_n)$.
Geometrically, this is a consequence of the constraints in~\eqref{eq:basic_ot_problem}---the feasible set is the intersection of an affine subspace (determined by linear constraints $\pi \bm 1 = \bm t, \pi^\top \bm 1 = \bm s$) with a polyhedral cone (determined by the nonnegativity constraint $\pi \geq \bm{0}$).
Small perturbations of the linear constraints can lead to abrupt changes in the geometry of this set, which in turn gives rise to non-differentiable changes in the optimal solution.

This lack of smoothness implies that $\pi_n$ is \textit{asymptotically biased},
in the classical sense that the limit of $\sqrt n(\pi^\star - \pi_n)$ is a non-centered random variable.
For instance, examining the northwest corner entry, we see that the limit of $\sqrt{n}((\pi_n)_{11} - \pi^\star _{11})$ is the minimum of two independent centered Gaussians, and therefore has negative mean.
Far from being specific to this example, asymptotic bias is in fact a general phenomenon plaguing the plan estimation problem: prior work~\citep{LiuBunNil23,KlaMunZem22} has shown that a limit of the form
\begin{equation*}
	\sqrt{n}(\pi_n - \pi^\star) \overset{d}{\to} h(\bm g)
\end{equation*}
for a Gaussian vector $\bm g$ and for a non-smooth function $h$.
The limiting random variable $h(\bm g)$ will only be centered in exceptional cases.

Previous studies of the statistical properties of the discrete optimal transport problem have sought to address this issue via regularization schemes~\citep{ferradans2014regularized,courty2014domain,essid2018quadratically}, the most prominent of which is entropic optimal transport, discussed in detail in Section \ref{sec: entropic} below. 
By adding a strictly convex regularizer, this approach guarantees that the optimal solution is a smooth function of the marginal distributions, thereby avoiding the pathologies of the plug-in estimator.
This regularization brings substantial statistical benefits: for example, under suitable assumptions on the regularizer $\varphi$, for a \textit{fixed} regularization parameter $\lambda > 0$, the solution $\pi_{\lambda, n}$ to the empirical regularized problem
\begin{equation}\label{eq:reg_ot_problem}
	\min_{\pi \in \RR^{p \times p}} \sum_{i,j=1}^p \pi_{ij}c_{ij} + \lambda \varphi(\pi) \quad \quad \text{s.t. } \pi \bm 1 = \bm t_n, \pi^\top \bm 1 = \bm s_n
\end{equation}
satisfies $\sqrt n(\pi_{\lambda, n} - \pi^\star_\lambda) \overset{d}{\to} \cN(0, \Sigma_{\lambda, \bm t, \bm s})$, for some $\Sigma_{\lambda, \bm t, \bm s} \in \mathbb R^{(p \times p) \times (p \times p)}$, where $\pi^\star_\lambda$ denotes the solution to the population level version of~\eqref{eq:reg_ot_problem}, with $\bm t_n$ and $\bm s_n$ replaced by $\bm t$ and $\bm s$, respectively.
The empirical estimator $\pi_{\lambda, n}$ therefore enjoys a centered Gaussian limit around the \textit{regularized} solution $\pi^\star_\lambda$, which \textit{does not agree} with $\pi^\star$ when $\lambda > 0$.

Existing regularization schemes therefore yield estimators with Gaussian limits, but at a price: these regularization schemes introduce an additional source of ``bias'' by replacing the original solution $\pi^\star$ with the solution to the modified program.
This is visible from the decomposition
\begin{equation*}
	\sqrt{n}(\pi_{\lambda, n} - \pi^\star) = \underset{\text{(I)}}{\underbrace{\sqrt{n}(\pi_{\lambda, n} - \pi^\star_\lambda)}} +\underset{\text{(II)}}{\underbrace{\sqrt{n}(\pi^\star_\lambda - \pi^\star)}} \,,
\end{equation*}
which shows that $\pi_{\lambda, n}$ will only be an asymptotically unbiased estimator of $\pi^\star$ if the both (I) and (II) can be controlled simultaneously.
In fact, as we show in Section~\ref{sec: entropic}, there is no reasonable way to tune the regularization parameter to obtain an estimator with good asymptotic properties: if $\lambda = \lambda_n$ tends to zero as $n \to \infty$ too slowly, then (II) is too large and the bias introduced by regularization is asymptotically dominant, whereas if $\lambda = \lambda_n$ tends to zero too quickly, then the benefits of regularization are lost and (I) no longer has a centered Gaussian limit.
In short, despite the success of regularization in the statistical analysis of optimal transport, all existing approaches give rise to asymptotically biased estimators of the optimal plan.

This paper introduces a novel approach to this problem.
We design a new class of regularization methods whose solutions can be easily \textit{debiased} to yield estimators $\hat \pi_n$ which enjoy Gaussian limits at the $\sqrt n$ rate, centered at the \textit{true} optimal plan:
\begin{equation}
	\sqrt n(\hat \pi_n - \pi^\star) \overset{d}{\to} \cN(0, \tilde \Sigma_{\lambda, \bm t, \bm s})\,.
\end{equation}
Moreover, the regularizers we propose are already popular in the context of interior point methods for convex optimization, so the solutions to our regularized programs can be computed quickly with off-the-shelf software.
We show in experiments on real and simulated data that this approach can be used to construct valid confidence sets in situations where existing methods fail.
To our knowledge, this represents the first method to construct asymptotically unbiased estimators and Gaussian confidence sets for optimal transport plans.

In fact, our approach applies in significantly more generality than the optimal transport problem: our theory applies to any ``standard form'' linear program
\begin{equation}\label{primal}
	\min_{\bm{x}\in\mathbb{R}^m} \langle\bm{c},\bm{x}\rangle,\qquad \text{s.t.}\ \bm{Ax}=\bm{b},\ \bm{x}\geq\bm{0},
\end{equation}
of which~\eqref{eq:basic_ot_problem} is, in historical terms, the earliest and most fundamental example~\citep[Historical notes]{schrijver1998theory}.
We consider the setting where $\bm{A} \in \mathbb{R}^{k\times m}$ and $\bm{c} \in \mathbb{R}^m$ are fixed, in which case the solution $\bx^\star$ to~\eqref{primal} is determined entirely by $\bb$.

Our interest is in the statistical setting where $\bb$ is replaced by a random counterpart $\bbn$, which satisfies, for some rate $q_n\rightarrow+\infty$,
\begin{equation}
	\label{converge bn}
	q_n(\bm{b}_n-\bm{b})\overset{d}\to{\bm G},
\end{equation}
for a non-degenerate random variable $\bm G$.
We construct estimators $\hat \bx_n$ such that
\begin{equation}\label{converge xn}
	q_n(\hat \bx_n - \bx^\star) \overset{d}{\to} \bm M^\star \bm G\,,
\end{equation}
for a deterministic matrix $\bm M^\star$.
In particular, as advertised above, when the limit distribution in~\eqref{converge bn} is a centered Gaussian, the same will be true of the limit in~\eqref{converge xn}.
This opens the door to asymptotically valid inference for any application of linear programming, for example, to the min-cost flow problem~\citep{ahuja1993network}.
As in the case of optimal transport, ours is the first valid asymptotic inference methodology with Gaussian limits for these applications.

We consider a class of regularized versions of the linear program~\eqref{primal} of the form
\begin{equation}
	\label{general penalty problem}
	\bm{x}(r,\bm{b})=\argmin_{\bm{x}\in\mathbb{R}^m} f_r(\bm{x}),\qquad \text{s.t.}\ \bm{Ax}=\bm{b},
\end{equation}
where $f_r(\bx)$ is a regularized version of the objective function and $r > 0$ is a tunable regularization parameter.
This gives rise to a penalized form of the empirical problem
\begin{equation}
	\label{ penalty problem with bn}
	{\bm{x}}(r_n,\bm{b}_n)=\argmin_{\bm{x}\in\mathbb{R}^m} f_{r_n}(\bm{x}),\qquad \text{s.t.}\ \bm{Ax}=\bm{b}_n\,,
\end{equation}
where $r_n$ is a carefully chosen sequence which converges to $0$ as $n \to \infty$.
Since both~\eqref{general penalty problem} and~\eqref{ penalty problem with bn} lack the nonnegativity constraint in~\eqref{primal}, $f_r$ will be chosen so that as $r \to 0$, the function $f_r(\bx)$ converges to $\langle \bc, \bx \rangle$ if $\bx > \bm 0$ but diverges to $+ \infty$ if any entry of $x$ is negative.
Intuitively, this property ensures that, when $r$ is small, solutions to~\eqref{general penalty problem} are close to those of the original linear program.

Our main insight is to choose the regularization in such a way that, under suitable assumptions on the decay of $r_n$, the solution $ x(r_n, \bb_n)$ of~\eqref{ penalty problem with bn} satisfies an asymptotic expansion of the form
\begin{equation}\label{eq:heuristic_asymptotic}
	 \bx(r_n, \bb_n) = \bx^\star + r_n \bd^\star + \bm M^\star(\bb_n - \bb) + \text{lower order terms}\,,
\end{equation}
for some deterministic vector $\bd^\star$ and matrix $\bm M^\star$.
The term $r_n \bd^\star$ can be viewed as the ``bias,'' of order $r_n$, introduced by the regularization scheme.
The crucial fact we rely on to develop our estimator is that this additional bias is \textit{linear} in the regularization parameter; this stands in sharp contrast to the case of entropic regularization, for example, for which the bias term is non-linear (see Section~\ref{sec: entropic}).

On its own,~\eqref{eq:heuristic_asymptotic} seems to indicate that achieving asymptotically unbiased estimation of $\bx^\star$ will require separately estimating the term $r_n \bd^\star$, perhaps via a more involved statistical procedure.
This is the approach traditionally taken in the literature on, for example, the debiased lasso~\citep{ZhaZha13}.
However, we exploit the linearity of the bias term to bypass the need to estimate it explicitly.
Given ${\bm{x}}(r_n,\bm{b}_n)$ defined by~\eqref{ penalty problem with bn}, our proposed estimator takes the form
\begin{equation}\label{eq:extrapolation_estimator}
	\hat \bx_n = 2 \bx(\tfrac{r_n}{2}, \bb_n) - \bx (r_n, \bb_n)\,.
\end{equation}
This trick---which has been rediscovered many times in the literature, for instance under the name ``Richardson extrapolation'' in numerical analysis~\citep{Ric11} and ``twicing'' in statistics~\citep{newey2004twicing,zhang2012twicing}---provides an easy solution for the problem of bias.
Indeed, using~\eqref{eq:heuristic_asymptotic}, we observe that $\hat \bx_n$ satisfies
\begin{equation*}
	\hat \bx_n = \bx^\star + \bm M^\star(\bb_n - \bb) + \text{lower order terms}\,.
\end{equation*}
Comparing this expression with~\eqref{eq:heuristic_asymptotic}, the bias term has completely disappeared.
Moreover, the fact that $\hat \bx_n$ is an asymptotically linear estimator implies that the estimator $\hat \bx$ enjoys centered Gaussian limits around the population-level quantity whenever $\bb_n$ does.
As we show in Section~\ref{sec: bootstrap}, this result opens the door to practical inference for $\bx^\star$ via the na\"ive boostrap.
We show in simulation that the resulting confidence sets are much more accurate than existing approaches for plan inference in the literature.

The remainder of the paper is structured as follows.
We review prior work in Section~\ref{sc:prior}.
We demonstrate some deficiencies of the entropic penalty in Section~\ref{sec: entropic} and use our analysis of the entropic penalty to motivate the penalization approach we develop in this work, which we specify precisely in Section~\ref{sec:preliminaries}.
To develop the asymptotic expansion~\eqref{eq:heuristic_asymptotic}, we first consider the case of where $\bb_n = \bb$ in Section~\ref{sec: trajectory of nonrandom}, then extend our results to the case of randomly perturbed programs in Section~\ref{sec: trajectory of random LP}.
Finally, we define our debiased estimator in Section~\ref{sec:trajectory with bn} and prove consistency of the bootstrap in Section~\ref{sec: bootstrap}.
Simulation and experimental results appear in Sections~\ref{sec:discussion} and~\ref{sec:numerical}.

\subsection{Prior work}\label{sc:prior}
There is a long history of studying the asymptotic properties of stochastic optimization problems~\citep[see, e.g.,][]{Sha91,Sha93,KinRoc93,PolJud92,DupWet88}.
In statistics, the classical large sample theory of maximum-likelihood estimation provides a paradigmatic example, reaching back to the 1940's~\citep{Wal43, Wal49}.
An important insight stemming from this work is that the asymptotic behavior of solutions depends heavily on the regularity assumptions imposed on optimization problem: in the most benign scenario (e.g., unconstrained smooth and strongly convex optimization problems), centered Gaussian limits are the norm, and while these conditions can be weakened somewhat~\citep{AitSil58,Hub67}, in general Gaussian limits are the exception rather than the rule~\citep{Sha89}.
Moreover, though Gaussian limits can arise in the solutions to non-classical problems~\citep{DavDruJia24, DucRua21}, this phenomenon is typically limited to the situation of an optimization problem with a \textit{random} objective function on a \textit{fixed} constraint set.

In light of this work, the random linear programs we study---with a non-strongly-convex objective and random constraints---represent a challenging edge case.
Prior work studying the asymptotic properties of solutions to random linear programs reveals that they possess Gaussian limits only in exceptional circumstances~\citep{KlaMunZem22,LiuBunNil23}.
In general, the limits are non-Gaussian and asymptotically biased.

The explosive interest in regularization schemes for optimal transport has been partially motivated by the search for estimators with better asymptotic properties.
Entropic regularization---along with its variants---gives rise to Gaussian limiting distributions for the optimal solution~\citep{KlaTamMun20}, a phenomenon that even extends to the case of continuous marginals~\citep{HarLiuPal24,GonLouNil22,GolKatRio24,GonHun23}, but they are {\it not} centered at the true optimal plan.
In some cases, under strong regularity assumptions, it is possible to carefully tune the regularization to obtain Gaussian limits centered at the true optimal plan~\citep{Mor24,ManBalNil23}, albeit not at the $\sqrt n$ rate.

Taking a linear combination of estimators with different regularization parameters, as in~\eqref{eq:extrapolation_estimator}, is related to a classical idea in numerical analysis known as ``Richardson extrapolation''~\citep{Ric11,Joy71}, which was rediscovered in the statistics literature by Tukey under the name ``twicing''~\citep{StuMit79,Tuk77}.
In recent years, \cite{Bac21} has highlighted the effectiveness of Richardson extrapolation in machine learning applications, including  regularizing convex optimization problems.
This perspective has been applied to obtaining better estimates for the cost of the optimal transport problem by~\cite{ChiRouLeg20}.

At a technical level, our results are most similar to prior work studying the ``convergence trajectories'' of penalized linear programs~\citep{1994asymptotic,AusComHad97}. In the particular case of the exponential penalty function, \cite{1994asymptotic} show an asymptotic expansion of the form~\eqref{eq:heuristic_asymptotic} when $\bbn = \bb$. \cite{AusComHad97} analyze a more general class of penalty functions, which include those studied in this work, but do not obtain a full expansion akin to~\eqref{eq:heuristic_asymptotic}.

\subsection{Notation}
We work in an asymptotic framework where $n \to \infty$ (and the regularization parameter $r = r_n \to 0$).
The symbols $O(\cdot)$ and $o(\cdot)$, as well as their probabilistic counterparts $O_p(\cdot)$ and  $o_p(\cdot)$, are to be understood in this setting.
We adopt the shorthand $a_n \ll b_n$ to mean $a_n = o(b_n)$.

Given a function $g: \mathcal{X} \to \RR \cup \{+ \infty\}$ for some $\mathcal{X} \subseteq \RR^d$, its convex conjugate $g^* : \RR^d \to \RR \cup \{+ \infty\}$ is defined by
\begin{equation*}
	g^* (y) = \sup_{x \in \mathcal{X}} \langle x, y \rangle - g(x)\,.\
\end{equation*}
The function $g^* $ is always convex and lower semi-continuous.
If $g$ shares these same properties, then $(g^*)^*  = g$.

Thoughout this work, we permit functions to take values in the extended reals $\RR \cup \{+ \infty\}$.
Given such a function $g$, we write $\dom g = \{x: g(x) < +\infty\}$ for the domain of $g$.
Given a convex set $X$, we write $\inter X$ for its interior and $\relint X$ for its relative interior (that is, its interior with respect to the subset topology on the affine span of $X$).

Throughout, our goal is to perform inference on the solution $\bx^\star$ to~\eqref{primal}, where $\bc$ and $\bA$ are fixed and known but we only have access to a random $\bbn$ in place of $\bb$.
We use star superscripts for quantities ($\bx^\star, \bd^\star, \bm M^\star$) that depend on the true, unknown vector $\bb$.

\section{The problem of bias in entropic optimal transport}
\label{sec: entropic}
The most popular form of regularized optimal transport is \textit{entropic optimal transport}, which corresponds to~\eqref{eq:reg_ot_problem} with the choice
\begin{equation}
	\varphi(\pi) = \sum_{i,j = 1}^p \pi_{ij} \log \pi_{ij}\,.
\end{equation}
Though originally motivated by computational considerations~\citep{ChiRouLeg20,peyre2019computational}, entropic optimal transport  is now known to bring significant statistical benefits as well~\citep{Mor24,GonHun23,GolKatRio24,KlaTamMun20}.
For example, \cite{KlaTamMun20} showed a central limit theorem for the empirical entropic optimal transport plan $\pi_{\lambda, n}$, \emph{centered at the population solution to the regularized program $\pi_\lambda^\star$}.
Their asymptotic results hold when the regularization parameter $\lambda$ is fixed and $n$ tends to infinity. 
Concretely, they show
\begin{equation*}
	\sqrt{n}(\pi_{\lambda,n} - \pi_\lambda^\star ) \overset{d}{\to} \mathcal{N}(0, \Sigma_{\lambda, \bm{t}, \bm{s}})
\end{equation*}
As Klatt et al. note \citep[Section 3.2]{KlaTamMun20}, this behavior is not expected in general: if $\lambda = \lambda_n$ tends to zero as $n \to \infty$ at a sufficiently fast rate, then the limit law will no longer be Gaussian.
On the other hand, if $\lambda = \lambda_n$ tends to zero more slowly, it is reasonable to conjecture that bias will dominate.

The following result makes this intuition precise.
We say a function $h(\lambda)$ decays exponentially fast to $0$ if $\lambda\log(1/h(\lambda))$ is bounded away from zero and infinity as $\lambda \to 0$.
For convenience, we impose the additional technical condition that $\pi^\star$ is non-degenerate (see proof for a precise definition) in the slow decay regime of $\lambda_n$, but similar conclusions apply to the general case as well.
\begin{proposition}
	\label{proposition: ROT}
	Fix a sequence $\lambda = \lambda_n \to 0$, and assume that $\sqrt n(\bm t_n - \bm t)$ is stochastically bounded away from zero and infinity. 
	If $\lambda \gg \frac{1}{\log n}$ and $\pi^\star$ is non-degenerate, then as $n \to \infty$ the entropy-regularized empirical optimal transport plan can be written as  
	\begin{equation}
		\label{equ: ROT expansion large regularizer}
		\sqrt n (\pi_{\lambda, n} - \pi^\star) = \sqrt n \bm{E}(\lambda)+\sqrt n \bm{M}(\bm{t}_n-\bm{t})+O_p(\epsilon(\lambda))
 	\end{equation}
	for a matrix $\bm{M}$, and where $\epsilon(\lambda)$ and $\bm{E}(\lambda)$ denote two terms that decay exponentially fast to zero (at rates which depend on $\pi^\star$).
	
	On the other hand, if  $\lambda \ll \frac{1}{\log n}$, then
	\begin{equation}
		\label{equ: ROT expansion}
		\sqrt n (\pi_{\lambda, n} -\pi^\star) =\sqrt n(\pi_n - \pi^\star) +O_p(\epsilon(\lambda)),
	\end{equation}
	where $\epsilon(\lambda)$ decays exponentially fast to zero.
\end{proposition}
To interpret the above result, we note that in the first regime, if $\lambda \gg \frac{1}{\log n}$, then the regularization term is relatively large, and the ``bias'' term $\bm E(\lambda)$ dominates.
In particular, since $\lambda \log(1/\bm E(\lambda)) \not \to \infty$, the term $\bm E(\lambda)$ satisfies
\begin{equation*}
	\sqrt n \bm{E}(\lambda) \to \infty \quad  \text{as $n \to \infty$,}
\end{equation*}
so $\sqrt n(\pi_{\lambda, n} - \pi^\star)$ cannot converge to a centered limit.

On the other hand, in the second regime, if $\lambda \ll \frac{1}{\log n}$, then the regularization term is small enough that the asymptotic limit of $\sqrt n (\pi_{\lambda, n} - \pi^\star)$ is the same as that of $\sqrt n (\pi_{n} - \pi^\star)$, so any asymptotic benefits of the regularization scheme are lost.

Though we do not rigorously characterize the behavior of $\pi_{\lambda, n}$ when $\lambda \asymp \frac{1}{\log n}$, it is reasonable to conjecture on the basis of our proof that~\eqref{equ: ROT expansion large regularizer} still holds in this case.
But note that this implies that whether the bias term $\bm E(\lambda)$ is asymptotically larger or smaller than $n^{-1/2}$ depends on the precise \textit{constant} that $\lambda \log n$ converges to, as well as the (unknown) rate at which $\bm E(\lambda)$ decays.
Tuning the regularization parameter this precisely is clearly outside the reach of a statistician interested in practical inference.

Our goal in this work is to design a different scheme, which is much less sensitive to regularization strength, and which naturally yields centered Gaussian limits.

\subsection{Designing our alternative scheme}\label{sec:alt_scheme}
As Proposition~\ref{proposition: ROT} makes clear, the central difficulty in using entropic optimal transport to obtain asymptotically unbiased estimators is the fact that the bias term $\bm E(\lambda)$ is not explicit.
We therefore consider alternative regularization schemes whose bias can be explicitly characterized.
As we shall see, the resulting regularizers are, in a precise sense, the convex duals of those that have been previously considered in the literature.

To motivate our proposal, we return to the key idea, mentioned in Section~\ref{sec:intro}, that our regularization should be such that the bias term is \textit{linear} in the regularization parameter.
We develop this idea for the simplest linear program: given a vector $\bc$ with positive coordinates, consider
\begin{equation}
	\min_{\bm{x}\in\mathbb{R}^m} \langle\bm{c},\bm{x}\rangle,\qquad \text{s.t.}\, \bm{x}\geq\bm{0},
\end{equation}
the optimal solution to which is clearly $\bx^\star = \bm{0}$.
We aim to design a regularized version of the same problem:
\begin{equation}\label{eq:simple_pen}
	\min_{\bm{x}\in\mathbb{R}^m} f_r(\bx)
\end{equation}
where $f_r(\bx) = \langle \bm c, \bm x \rangle + \psi_r(\bx)$ for a strictly convex function $\psi_r$ to be chosen.

First-order optimality conditions imply that the solution $\bx_r$ to~\eqref{eq:simple_pen} satisfies
\begin{equation*}
	\bm c + \nabla \psi_r(\bx_r) = 0\,,
\end{equation*}
leading to the explicit expression
\begin{equation}\label{eq:bxr}
	\bx_r = (\nabla \psi_r)^{-1}(- \bm c)\,,
\end{equation}
where the invertibility of $\nabla \psi_r$ is guaranteed by the strict convexity of $\psi_r$.
Since $\psi_r$ is convex, its inverse gradient satisfies $(\nabla \psi_r)^{-1} = \nabla \phi_r$, where we have defined $\phi_r = \psi_r^* $ to be the convex conjugate of $\psi_r$.
The desired linearity of the bias term is therefore equivalent to the requirement that as $r \to 0$
\begin{equation}\label{eq:desired_expansion}
	\bx_r - \bx^\star = \nabla \phi_r(-\bm c) = r \bd^\star + o(r)\,,
\end{equation}
where $\bd^\star $ is some vector depending on $\bm c$ but not on $r$.
The simplest means of guaranteeing this expansion is to simply assume that $\phi_r$ is a linear function of $r$; that is, that
\begin{equation}\label{eq:linear_ansatz}
	\phi_r(- \bm c) = r \varphi(\bm c)
\end{equation}
for some fixed convex function $\varphi$, so that~\eqref{eq:desired_expansion} holds with $\bd^\star = - \nabla \varphi(-\bm c)$.
Since convex duality implies that $\psi_r = \phi_r^* $, we therefore obtain
\begin{align*}
	\psi_r(\bx) &= \sup_{\bm y} \langle \bx, \bm y \rangle - \phi_r( \bm y) \\
	& = \sup_{\bm y} \langle \bx, \bm y \rangle -r \varphi(- \bm y) \\
	& = r \sup_{\bm y} \langle - \bx/r, \bm y \rangle - \varphi(\bm y) \\
	& = r \varphi^* (-\bx/r)\,.
\end{align*}

The above \textit{ansatz} is the heart of our approach.
We will assume that our regularizer takes the form $r \varphi^* (-\bx/r)$ for a convex function $\varphi$.
For computational simplicity, we will further assume that $\varphi$ is separable, so that $\varphi(\bm y) = \sum_{i=1}^m q(y_i)$ for some univariate convex $q$.
Under these assumptions, the objective of our regularized program~\eqref{general penalty problem} ultimately reads
\begin{align}
	\bm{x}(r,\bm{b}) & =\argmin_{\bm{x}\in\mathbb{R}^m} f_r(\bm{x}),\qquad \text{s.t.}\ \bm{Ax}=\bm{b}, \nonumber \\
	& \phantom{=} \text{ where } f_r(\bm{x}):=\langle\bm{c},\bm{x}\rangle + r\sum_{i=1}^m p(-\frac{x_i}{r}), \label{equ:penalized function}
\end{align}
for a convex function $p := q^* $.
This is the form of regularization that we study in what follows.

\begin{remark}\label{remark:dual_of_existing}
	Prior work on regularization for optimal transport~\citep{KlaTamMun20} has focused almost exclusively on regularizers of the form $\lambda \varphi(\pi)$ for some convex $\phi$, as in~\eqref{eq:reg_ot_problem}.
	Recalling~\eqref{eq:linear_ansatz}, we observe that the regularizations we study are in fact the convex conjugates of those typically considered in the optimal transport literature.
	Our results reveal that the effect of this ``dual regularization'' is very different from that of standard regularizers, and necessitates the development of new proof techniques.
	In fact, our new techniques also yield insight into traditional regularizers such as the entropy function, see Proposition~\ref{proposition: ROT}.
\end{remark}

\section{Technical preliminaries}\label{sec:preliminaries}
Throughout the paper, we make certain mild assumptions on the random vector $\bm{b}_n$, the linear program, and the choice of penalty function.
We detail these assumptions below.
\subsection{Assumption on the random vector}
\begin{assumption}\label{assumption:G}
	The limiting random variable in~\eqref{converge bn} satisfies $P({\bm G}=\bm 0)=0$.
\end{assumption}
This assumption is automatically satisfied in our primary setting of interest, when $\bm G$ is a non-degenerate Gaussian.
In particular, Assumption~\ref{assumption:G} guarantees that $q_n$ is the proper scaling for convergence of $\bm{b}_n$ to $\bb$.

\subsection{Assumptions on the linear program}

\begin{assumption}
    The linear program satisfies the following conditions:
    \label{assumption on LP}
    \begin{enumerate}
         \item The solution to~\eqref{primal} exists and is unique.
        \item \emph{(Slater's condition)} The feasible region $\mathcal{D}=\{\bx: \bA \bx = \bb, \bx \geq \bz\}$ of~\eqref{primal} has non-empty relative interior and there exists $\bm{x}\in \relint \mathcal{D}$ such that $\bm{x}>0$. 
        \item The constraint matrix $\bA$ in~\eqref{primal} has full row rank.
    \end{enumerate}
\end{assumption}
These assumptions are not restrictive and are natural in this setting.
The assumption that~\eqref{primal} has a unique solution is true of most linear programs that arise in practice\footnote{Indeed, uniqueness holds for all but a measure-zero set of costs $\bc$.} and is necessary to be able to formulate standard distributional limit results for the stochastic analogue of~\eqref{primal}.
Likewise, Slater's condition guarantees that the solution to~\eqref{primal} remains feasible under small perturbations of $\bb$, which guarantees that~\eqref{ penalty problem with bn} is well defined.
Finally, the assumption that $\bm{A}$ has full rank is without loss of generality, as redundant constraints can always be removed without changing the program. 

\subsection{Valid penalty functions}
The penalty function $p$, which we always assume to be convex and lower semi-continuous, plays a central role in this work.
As Section~\ref{sec:alt_scheme} makes clear, it is natural to make assumptions on $q$, where $p = q^*$.

\begin{assumption}\label{new assumptions on p}
	The penalty function $p: \Rset \to \Rset \cup \{+ \infty\}$ is given by $p = q^* $, where $q: (0, \infty) \to \Rset$ is twice differentiable, has locally Lipschitz and strictly positive second derivative, and satisfies $\lim_{x \to 0} q'(x)=-\infty$,  $\lim_{x \to \infty} q'(x)\geq 0$
	
\end{assumption}
In words, we require that $q$ is a strictly convex barrier function for the set $(0, \infty)$.
As we shall see, this property of $q$ implies that the dual of~\eqref{general penalty problem} can be written without inequailty constraints, which significantly simplifies the analysis.

This assumption implies the following properties of $p$.

\begin{lemma}\label{lem:p_properties}
	Under Assumption~\ref{new assumptions on p}, the following properties hold:
	\begin{enumerate}
		\item The domain of $p$ contains $(-\infty, 0)$.
        \item $p'(x)$ is unbounded above. 
		\item $\lim_{x \to -\infty} p'(x) = 0$.
		\item $p''$ is strictly positive and locally Lipschitz on its domain.
	\end{enumerate}
\end{lemma}
In the following sections, we will need to choose the sequence of penalty coefficient $r_n$ based on how fast the derivative of the penalty function decays to zero at $-\infty$. Therefore, to facilitate the discussion, we define the decay rate of a function at $-\infty$:
\begin{definition}[Decay rate of a function at $-\infty$]
\label{decay rate}
Let $\beta: \Rset \to \Rset$ be any continuous function such that $\lim_{r \to 0} \beta(r) = 0$.
We say a function $f$ decays to zero at $-\infty$ at rate $\beta(r)$ if

\begin{equation}
    \lim_{r\downarrow 0} \frac{f(-\frac{1}{r})}{\beta(r)} <\infty\,.
\end{equation}

\end{definition}

Table~\ref{tab:functions} lists several examples of penalty functions that satisfy the assumptions: the log-barrier function, inverse polynomial functions, smoothed quadratic penalty function, and the exponential function.
The admissibility of the log-barrier function is particularly notable, since penalties of this type are already widely used in interior point methods for linear programming~\citep{BoyVan04}.








\begin{figure}[ht]
    \centering
\begin{minipage}{0.65\textwidth}
\centering
   
\begin{tabular}{ |c|c|c| }

\hline
Name& Domain & Decay rate of $p'(x)$ \\
\hline
log-barrier function  &\multirow{2}*{ $(-\infty,0)$} & \multirow{2}*{ $r$} \\
$p(x)=-\ln(-x)$&~&~\\
~&~&~\\

inverse polynomial functions &\multirow{2}*{ $(-\infty,0)$ }&  \multirow{2}*{$r^{\alpha+1}$}\\
$p(x)=\frac{1}{|x|^{\alpha}},\ \alpha>0$&~&~\\
~&~&~\\
smoothed quadratic penalty function &\multirow{2}*{ $(-\infty,\infty)$ }&  \multirow{2}*{Any polynomial of $r$} \\
$p(x)=(\log(1+\exp(x)))^2$&~&~\\
~&~&~\\
exponential function & \multirow{2}*{$(-\infty,\infty)$ } & \multirow{2}*{Any polynomial of $r$}\\
$p(x)=\exp(x)$&~&~\\

\hline

\end{tabular}
\end{minipage}
\hfill
\begin{minipage}{0.30\textwidth}
    \centering
   \includegraphics[width=\textwidth]{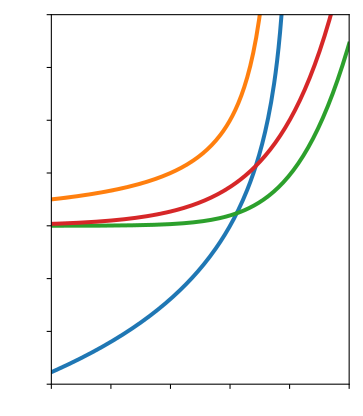}
\end{minipage}
\caption{Examples of proper penalty functions (table on the left) and plot of these functions (figure on the right): log-barrier function (blue), inverse polynomial function with $\alpha=1$ (yellow), smoothed quadratic penalty function (green), exponential function (red).  }
    \label{tab:functions}
\end{figure}
%

%

\section{Convergence of trajectory of the penalized linear program}\label{sec: trajectory of nonrandom}
In this section, we focus on the ``population'' version of the penalized linear program, with $\bbn = \bb$.
Our goal is to precisely characterize the bias induced by the regularization term as a stepping stone to~\eqref{eq:heuristic_asymptotic}.
Specifically, we study the relationship between~\eqref{primal} and~\eqref{general penalty problem} as $r \to 0$, and the main result of this section is that $\bx(r, \bb) - \bx^\star$ is asymptotically linear in $r$.
This extends a result of \cite{1994asymptotic}, who showed an analogous claim in the special case of the exponential penalty $p(x) = \exp(x)$.
An important feature of our result, as compared with that of~\cite{1994asymptotic}, is that we identify the magnitude of the error term, and show that it depends on the rate of decay of $p'$ at $-\infty$ given by Definition \ref{decay rate}. 

We first show that that the solution to~\eqref{general penalty problem} is unique.
\begin{proposition}
\label{existence and uniqueness}
Under Assumptions~\ref{assumption on LP} and~\ref{new assumptions on p}, the penalized program~\eqref{equ:penalized function} has a unique solution, which satisfies $\bx(r, \bb) \in \inter \dom f_r$.
\end{proposition}



The following theorem provides a first-order expansion of the solution to the penalized program~\eqref{general penalty problem} around the solution $\bx^\star$ to~\eqref{primal} of the form $\bm{x}(r,\bm{b})=\bm{x}^\star+r\bm{d}^\star+o(r)$, where $\bd^\star$ is the solution to an auxiliary optimization problem.

Denote by $I_0:=\{i\in[m]| x^\star_i=0\}$ the set of zero entries in the optimal solution $\bx^\star$.
\begin{theorem}[Convergence of trajectory of the penalized program]
    \label{theorem: unique solution general penalty}
Under Assumptions~\ref{assumption on LP} and~\ref{new assumptions on p}, the solution to~\eqref{general penalty problem} can be written as
\begin{equation}
\bm{x}(r,\bm{b})=\bm{x}^\star+r\bm{d}^\star+O(r\beta(r)) \quad \text{as $r \to 0$,}
\end{equation}
where $\bx^\star$ is the true optimal LP solution, $\beta(r)$ is the decay rate of $p'$ at $-\infty$, and $\bm{d}^\star$ is the unique solution to 
\begin{equation}
\label{equ for 1st order}
    \min_{\bm{d}\in\mathbb{R}^m} \langle \bm{c},\bm{d} \rangle+\sum_{i\in I_0} p(-d_i), \quad \text{s.t. } \bm{Ad}=0.
\end{equation}
\end{theorem}

The proof is based on the following observation.
Write $\bd(r) =(\bm{x}(r, \bb)-\bm{x}^\star)/r$.
Then, since $\bx(r, \bb)$ minimizes~\eqref{general penalty problem} and $\bA \bx^\star = \bb$, we obtain that
\begin{equation}\label{eq:delta_char}
	\bd(r) = \argmin_{\bd \in \Rset^m} g_r(\bd)\,, \quad \quad \text{s.t. } \bA \bd = \bz\,,
\end{equation}
where we define
\begin{equation}\label{eq:gr_def}
	g_r(\bd) =  \langle \bc, \bd \rangle + \sum_{i=1}^m p(-\frac{x_i^\star}{r} - d_i)\,.
\end{equation}
Suppose that we adopt the additional assumption that $p(-x) \to 0$ as $x \to \infty$.
Then $\lim_{r \to 0} g_r(\bd) = \langle \bc, \bd \rangle + \sum_{i \in I_0} p( - d_i)$.
Therefore, comparing~\eqref{eq:delta_char} with~\eqref{equ for 1st order}, we observe that~\eqref{equ for 1st order} is the $r \to 0$ limit of~\eqref{eq:delta_char}.
This suggests, that $\bd(r) \to \bd^\star$, so that $\bx(r, \bb) = \bx^\star + r \bd^\star + o(r)$.
The full proof of this theorem in Appendix~\ref{sec:theorem_proof} shows that this intuition holds even if $p(x) \not \to 0$ as $x \to \infty$, and shows that the error term depends on the behavior of $p'$ at $-\infty$ quantified by the decay rate $\beta$.
%
%
%


\section{Asymptotic expansion of the solution to the random penalized program}\label{sec: trajectory of random LP}
The goal of this section is to extend Theorem~\ref{theorem: unique solution general penalty} to the setting where the vector $\bb$ is replaced by a random counterpart $\bb_n$.
The following theorem is deterministic, and gives an asymptotic expansion for the solution $\bx(r, \bb')$, where $\bb'$ is a small perturbation of $\bb$ such that $\bb' \to \bb$ as $r \to 0$.
We then use this result to obtain a limit law for $\bx(r_n, \bb_n)$ where $\bb_n$ satisfies~\eqref{converge bn} and $r_n$ is a sequence of penalization parameters with prescribed decay.

We first establish that the penalized program remains feasible under small perturbations.
\begin{proposition}\label{prop:peturb_feasibility}
	There exists a neighborhood $U$ of $\bb$ such that for any $r > 0$, the penalized program~\eqref{general penalty problem} is feasible and has a unique solution when $\bb$ is replaced by any $\bb' \in U$.
\end{proposition}

To state our result, we need another auxiliary program.
Define $\bm \Sigma \in \RR^{m \times m}$ to be a diagonal matrix satisfying 
\begin{equation}
	\bm \Sigma_{ii} = p''(-{d}^\star_i)\mathbbm{1}_{\{i\in I_0\}}\,,
\end{equation}
where $\bm{d}^\star$ and $I_0$ are defined in~\eqref{equ for 1st order}.

Consider the quadratic program
\begin{equation}
\label{eq:mdef}
    \min_{\bm{x}\in\mathbb{R}^m} \bx^\top \bm \Sigma \bx, \text{ such that } \bm{Ax}=\bm{y},
\end{equation}
\begin{proposition}
\label{proposition: unique solution of control equation}
    There exists a unique solution to~\eqref{eq:mdef}, which is a linear function of $\bm{y}$, and can therefore be written $\bm{M}^\star\bm{y}$ for some $\bm{M}^\star\in \RR^{m \times k}$.

    Moreover, the matrix $\bm{M}^\star$ has the following properties:
    \begin{itemize}
        \item $\bm{M}^\star$ is a  right-inverse  of  $\bm{A}$, i.e., $\bm{A}\bm{M}^\star=I_k$. 
        \item $\ker(\bA) \subseteq \ker({\bm{M}^\star}^\top \bm \Sigma)$.
    \end{itemize}
    
\end{proposition}

The following result is our main technical result: a rigorous version of the asymptotic expansion promised in~\eqref{eq:heuristic_asymptotic}.

\begin{theorem}
	\label{thm: trajectory with noise}
	Under Assumptions~\ref{assumption on LP} and~\ref{new assumptions on p}, if $r\beta(r)\ll\|\bb - \bb'\|\ll r$,  then
	\begin{equation}\label{eq:rig_asymptotic}
		\bm{x}({r,\bb'})=\bm{x}^\star+r\bm{d}^\star+\bm{M}^\star (\bb' - \bb)+O\left(\frac{\|\bb' - \bb\|^2}{r} + r\beta(r)\right), \quad \text{as $r \to 0$,}
	\end{equation}
	where $\beta(r)$ and $\bd^\star$ are as in Theorem~\ref{theorem: unique solution general penalty} and $\bm{M}^\star$ is defined as in Proposition~\ref{proposition: unique solution of control equation}.
\end{theorem}
Comparing Theorem~\ref{thm: trajectory with noise} to Theorem~\ref{theorem: unique solution general penalty}, we see that $\bx(r, \bb')$ has the same first-order ``bias'' term $r \bd^\star$ as $\bx(r, \bb)$, and that the second order term is asymptotically linear in the perturbation $\bb' - \bb$.
In particular, if $\bb' - \bb$ is asymptotically Gaussian, then the fluctuations of $\bx(r, \bb')$ will have the same property.

Like Theorem~\ref{theorem: unique solution general penalty}, Theorem~\ref{thm: trajectory with noise} is based on analysis of~\eqref{eq:delta_char}.
If we write $g(\bm d) = \lim_{r \to 0} g_r(\bm d)$, for $g_r$ defined in~\eqref{eq:gr_def}, then the matrix $\bm \Sigma$ is the Hessian of $g$ at $\bm d^\star$.
It is therefore natural to conjecture that solutions to perturbations of the equation~\eqref{equ for 1st order} around $\bd^\star$ will be minimizers of a quadratic form involving $\bm \Sigma$.
The proof of Theorem~\ref{thm: trajectory with noise} establishes a quantitative form of this argument.

\begin{remark}\label{remark:size of r}
	To be explicit, Theorem~\ref{thm: trajectory with noise} holds in the asymptotic framework where $r \to 0$ and $\bb' \to \bb$ simultaneously, at a specified rate.
	The requirement that $\|\bb - \bb'\| \ll r$ is necessary for the expansion in~\eqref{eq:rig_asymptotic} to hold.
	Indeed, following the proof of Theorem~\ref{theorem: unique solution general penalty}, it can be shown that if $r\ll \|\bb - \bb'\|$, the solution to the penalized program can be written as 
	\begin{equation}
		\bm{x}(r,\bb')=\bm{x}^\star_{\bb'}+r\bm{d}^\star_{\bb'}+O\left(r\beta(r)\right),
	\end{equation}
	where $\bm{x}^\star_{\bb'}$ is the solution to the linear program with equality constraints $\bA \bx = \bb'$ and $\bm{d}^\star_{\bb'}$ is the solution to~\eqref{equ for 1st order} with $I_0$ replaced by the set of zero entries of $\bm{x}^\star_{\bb'}$. 
	Prior work~\citep{KlaMunZem22,LiuBunNil23} shows that $\bx^\star_{\bb'}$ is asymptotically non-linear in $\bb' - \bb$, so no expansion similar to~\eqref{eq:rig_asymptotic} holds in this regime.
\end{remark}

With Theorem~\ref{thm: trajectory with noise} in hand, we can obtain a distributional limit result for $ \bx(r_n, \bb_n)$ under the assumption that $\bb_n$ is random.
In light Remark~\ref{remark:size of r}, the following assumption is natural.
\begin{assumption}
	\label{assumptions on rates}
	The sequence $r_n$ is chosen to satisfy the relation $r_n\beta(r_n)\ll q_n^{-1}\ll r_n$, where $\beta(r)$ is the decay rate of $p'$ at $-\infty$.
\end{assumption}
Since $\bm{b}_n-\bm{b}\BN = \EN O_p(q_n^{-1})$  and Theorem~\ref{theorem: unique solution general penalty} shows that the effect of the penalty term is proportional to $r$, this assumption  is equivalent to requiring that the random fluctuations in $\bm{b}_n$ are smaller than the penalty  term but larger than the residual term, which has order $O(r\beta(r))$. 

Under this assumption, the following is a direct corollary of Theorem~\ref{thm: trajectory with noise}.
\begin{theorem}
\label{theorem: limit law}
Under Assumptions~\ref{assumption on LP},~\ref{new assumptions on p} and~\ref{assumptions on rates},
    \begin{equation}
        q_n({\bm{x}}(r_n,\bm{b}_n)-r_n\bm{d}^\star-\bm{x}^\star)\overset{d}\to\bm{M}^\star \bm G,
    \end{equation}
    where $\bd^\star $ and $\bm M^\star $ are as in Theorem~\ref{thm: trajectory with noise}.
\end{theorem}
Note that the limit law in Theorem~\ref{theorem: limit law} involves the bias term $r_n \bd^\star$, which depends on the true linear program solution $\bm{x}^\star$. In Section~\ref{sec:confidence set}, we introduce a debiasing procedure which removes this additional term.
\section{Debiasing and distributional limits of random linear program solutions with penalties }
\label{sec:trajectory with bn}
\label{sec:confidence set}

In this section, we exploit the linearity of the leading-order bias term to construct an asymptotically unbiased estimator of $\bx^\star$.
Our main insight is that the linearity of the bias makes it possible to construct a linear combination of the solutions to~\eqref{ penalty problem with bn} with two different values of $r_n$ so that the term involving $\bd^\star$ exactly cancels.
Indeed, the following theorem is immediate from Theorem~\ref{theorem: limit law}.
\begin{theorem}
\label{theorem: debiasing}
    Let ${\bm{x}}(r_n,\bm{b}_n)$ and ${\bm{x}}(\frac{r_n}{2},\bm{b}_n)$ be the solution to the problem~\eqref{ penalty problem with bn} with penalties $r_n$ and $\frac{r_n}{2}$, respectively.


    Define $\hat {\bm{x}}_n$ as
    \begin{equation}\label{equ:extrapolation formulation}
        \hat {\bm{x}}_n=2{\bm{x}}(\frac{r_n}{2},\bm{b}_n)-{\bm{x}}(r_n,\bm{b}_n).
    \end{equation}

    Then under Assumptions~\ref{assumption on LP},~\ref{new assumptions on p} and~\ref{assumptions on rates},
   \begin{equation}
        q_n(\hat \bx _n -\bm{x}^\star)\overset{d}\to\bm{M}^\star \bm G\,,
    \end{equation}
    where $\bm M^\star$ is as in Theorem~\ref{thm: trajectory with noise}.
\end{theorem}

In addition to being asymptotically unbiased, the estimator $\hat \bx_n$ sometimes gives rise to estimates of the \textit{objective value} that converge strictly faster than the plug-in estimate.
The following result develops this phenomenon in the particular context of optimal transport problems when $\bm{s} = \bm{t}$.
(See Corollary~\ref{corollary: fast convergence of cost} in the appendix for the extension to general linear programs.)
\begin{corollary}[Special case of Corollary~\ref{corollary: fast convergence of cost}]\label{corollary: OT fast convergence}
	Consider the discrete optimal transport problem~\eqref{eq:basic_ot_problem}.
	Suppose that $\bm{s}=\bm{t}$ , $c_{i,i}=0$, $c_{i,j}>0$ for $i\neq j$, and  $\bm{c}=\bm{c}^\top$.
	Let $\bm s_n$ and $\bm t_n$ be such that $\sqrt n (\bm s_n - \bm s)$ and $\sqrt n (\bm t_n - \bm t)$ converge to non-degenerate random variables.
	Let $\hat \pi_n$ be the debiased estimator constructed as in Theorem~\ref{theorem: debiasing} with the exponential penalty function $p(x) = \exp(x)$.

Then
    \begin{equation}
        \sqrt{n}(\langle\bm{c},\hat{\bm{\pi}}_n\rangle-\langle\bm{c},{\bm{\pi}^\star}\rangle)\overset{d}\to 0.
    \end{equation} 
\end{corollary}
 
The fast convergence identified in Corollary~\ref{corollary: OT fast convergence} occurs at the ``null'' when the two marginal distributions are equal to each other.
The Sinkhorn divergence~\citep{pmlr-v84-genevay18a}, a loss connected to entropic penalization, is known to evince a similar phenomenon~\citep[Theorem 2.7]{10.1214/19-EJS1637}.

\section{Bootstrap}\label{sec: bootstrap}
As emphasized above, the fact that the unique solution $\bx^\star$ to~\eqref{primal} is not a smooth function of $\bb$ (when viewing $\bA$ and $\bc$ as fixed) is the source of the asymptotic bias in the plug-in estimator $\bx_n$ obtained by solving~\eqref{primal} with $\bb$ replaced by its empirical counterpart $\bbn$.
More precisely, \cite{KlaMunZem22} show that $\bx^\star$ is in general only \textit{directionally} Hadamard differentiable with respect to $\bb$.
An implication of this fact is that the na\"ive  bootstrap is not consistent~\citep{Dum93,FanSan18} and does not lead to asymptotically valid inference for $\bx$ when applied with the plug-in estimator.

However, we show in this section that our debiased estimator from Theorem~\ref{theorem: debiasing} avoids this problem and can be bootstrapped without issue.
The resulting procedure is a practical and fully data driven method for asymptotic inference on $\bx$.

To be more precise, we consider the setting where the random vector $\bbn$ is an average of $n$ i.i.d.\ vectors $Z_1, \dots, Z_n$ with expectation $\bb$.
In particular, this is the case for the discrete optimal transport problem.
We denote by $\tilde \bb_n$ a bootstrap copy of $\bbn$ obtained by resampling with replacement from $\{Z_1, \dots, Z_n\}$.
Finally, we form the estimator $\hat \bx_n$ as in Theorem~\ref{theorem: debiasing} from $\bbn$, and use the same procedure to obtain a bootstrap version $\tilde \bx_n$ computed with $\tilde \bb_n$ in place of $\bbn$.
The following theorem shows that $\tilde \bx_n$ is consistent for inference on $\bx^\star$.
%
\begin{theorem}
\label{theorem: bootstrap}
Under the same assumptions as in Theorem~\ref{theorem: limit law}, let $\hat \bx_n$ and  $\tilde \bx_n$ be the estimator defined in Theorem~\ref{theorem: debiasing} with vectors $\bbn$ and $\tilde \bb_n$ respectively.
Then
\begin{equation}
\label{equation: theorem bootstrap}
     \sup_{h \in BL_1(\mathbb{R}^k)} \left| \mathbb{E}\left[h\left(\sqrt{n}(\tilde \bx_n - \hat \bx_n)\right) | Z_1, \ldots, Z_n\right] - \mathbb{E}\left[h\left(\sqrt{n}(\hat \bx_n-\bm{x}^\star )\right)\right] \right|\overset{p}\to 0\,,
\end{equation}
   where $BL_1((\mathbb{R}^k))$ denotes the set of bounded Lipschitz functions on $\RR^k$, whose $L^\infty$ norm and Lipschitz constant are both bounded by $1$.
\end{theorem}
As is well known, such conditional weak convergence results lead to consistency for quantile estimators and confidence sets obtained from $\tilde \bx_n$.
For example, the following result is a direct application of~\cite[Lemma 23.3]{Vaa98} showing that bootstrap estimates readily yield valid confidence intervals for arbitrary coordinates of the optimal solution.
For any $i \in [m]$ and $t \in (0, 1)$, let $\widehat F_{n, i}^{-1}(t)$ be the corresponding quantile of $\sqrt n ((\tilde \bx_n)_i - (\hat \bx_n)_i)$ given $Z_1, \dots, Z_n$.
\begin{corollary}\label{cor:bootstrap_quantiles}
	Under the same assumptions as Theorem~\ref{theorem: bootstrap}, for any $\alpha \in (0, 1)$,
	\begin{equation}
        \liminf_{n\rightarrow\infty} \mathbb{P}\left((\hat \bx_n)_i - \frac{\widehat F_{n, i}^{-1}(1 - \alpha/2)}{\sqrt n}\leq \bx^\star_i \leq (\hat \bx_n)_i - \frac{\widehat F_{n, i}^{-1}(\alpha/2)}{\sqrt n}\right)= 1-\alpha\,.
	\end{equation}
\end{corollary}

\section{Simulation results for the optimal transport problem}\label{sec:discussion}
A central application of our work is to providing asymptotically unbiased (and asymptotically Gaussian) estimators for the optimal plan in the discrete optimal transport problem.
Motivated by this application, we perform simulations to assess the performance of our estimator on optimal transport problems, mimicking an experimental setup for the entropic penalty due to~\cite{KlaTamMun20}.

\subsection{The $2\times 2$ OT problem revisited}
We consider a $2$-by-$2$ optimal transport problem:
\begin{equation}
   w^\star =\min_{\pi \in \mathbb{R}^{2 \times 2}} \pi_{12} + 2\pi_{21}\,, \qquad  \textrm{s.t.}\  {\pi} \bm{1} =\bm{t}, {\pi}^\top \bm{1} = \bm{s}, \pi \geq \bm{0}.
\end{equation}
with $\bm{t}=\bm{s}=\left [\frac{1}{2},\frac{1}{2}\right]^\top$ and a unique solution ${\pi}^\star=\begin{bmatrix}
    \frac{1}{2} & 0 \\
    0 & \frac{1}{2}\\
\end{bmatrix}$, and in the  statistical setting where $\bm t$ and $\bm s$ are only available via empirical frequencies $\bm{t}_n$ and $\bm{s}_n$.
This is the same setting as in~\eqref{equ: 2by2 example}, though we have used an asymmetric cost to avoid the degenerate limiting situation described in Corollary~\ref{corollary: OT fast convergence}.
It is easy to check that the asymptotic characterization of the plug-in estimator given in Proposition~\ref{prop:2by2solution} holds for this asymmetric cost as well.
In particular, the limiting distribution of $\pi_n$ is non-Gaussian.

Our approach in this paper suggests solving a penalized optimal transport problem:
\begin{equation}
   \pi((\bm{t}_n, \bm{s}_n), r_n)= \argmin_{\pi \in \mathbb{R}^{2 \times 2}} \pi_{12} + 2\pi_{21}+r_n\sum_{i,j=1}^{2} p(-\frac{\pi_{i,j}}{r_n}) \,, \qquad  \textrm{s.t.}\  {\pi} \bm{1} =\bm{t}_n, {\pi}^\top \bm{1} = \bm{s}_n,
\end{equation}
and using $\hat{{\pi}}_{n}=2\pi((\bm{t}_n, \bm{s}_n), r_n/2) - \pi((\bm{t}_n, \bm{s}_n), r_n)$ instead of ${\pi}_n$  as an estimator for ${\pi}^\star$.

We first consider the case of the log penalty $p(x) = - \log(-x)$.
Since $p'$ decays at the rate $\beta(r) = r$ at $-\infty$, Theorems~\ref{thm: trajectory with noise} and~\ref{theorem: debiasing} imply that as long as the sequence $r_n$ satisfies $r_n^2 \ll n^{-1/2} \ll r_n$, then
\begin{equation}
	\label{equ: toy example theory}
	\hat{{\pi}}_{n}={\pi}^\star+\bm{M}^\star(\bm{t}_n,\bm{s}_n)+O_p(\max(r_n^2,\frac{1}{nr_n}))\,.
\end{equation}
The error term is minimized when $r_n$ is of the order $n^{-1/3}$, so we choose $r_n=\frac{r_0}{\sqrt[3]{n}}$ for a positive constant $r_0$, and we write $\hat{{\pi}}_{n, r_0}$ to emphasize that the estimator depends on the choice of $r_0$.

By Theorem~\ref{theorem: debiasing}, the estimator of the transportation cost $\hat w_{n, r_0} := \langle \bm{c},\hat \pi_{n, r_0} \rangle$ satisfies
\begin{equation}
    \sqrt{n}(\hat w_{n, r_0} - w^\star)\overset{d}\to \langle \bm{c},\bm{M}^\star (\bm{G}_t, \bm{G}_s)\rangle := \bm{G}_{w},
\end{equation}
where $w^\star = \langle \bc, \pi^\star \rangle$ and $\sqrt n((\bm{t}_n, \bm{s}_n) - (\bm t, \bm s)) \overset{d}{\to} (\bm{G}_t, \bm{G}_s)$.
Using~\eqref{eq:model} and the definition of $\bm{M}^\star$ in~\eqref{eq:mdef}, one can show that $\bm{G}_w \sim\mathcal{N}(0, \frac{1}{8})$.

We conduct experiments with varying $n$ and varying $r_0$ and compute the rescaled cost
\begin{equation}
\label{equ: rescaled cost}
    \Delta w_{n,r_0}\coloneqq\sqrt{n/\text{var}(\bm{G}_{w})}(\hat{w}_{n, r_0}-w^\star)
\end{equation}
for $1000$ independent trails.
Our theoretical results indicate that $\hat w_{n, r_0}$ converges to $w^\star$ and that
$\Delta w_{n, r_0}$ converges in distribution to a standard Gaussian.
To assess this convergence, we plot both the mean squared error $\|\hat w_{n, r_0} - w^\star\|^2$ and the Kolmogorov--Smirnov distance between the distribution of $\Delta w_{n, r_0}$ and the standard Gaussian (Figure~\ref{fig: toy example result}\subref{subfig: log-pen-density }, top two rows).
The simulation reveals that the distributional behavior of the estimator is sensitive to the choice of $r_0$, with substantial deviations from Gaussianity when $r_0$ is too small or too large.
This phenomenon can be understood by recalling the condition $r_n^2 \ll n^{-1/2} \ll r_n$, or, equivalently, $n^{-1/2} \ll r_n \ll n^{-1/4}$.
When $n$ is small, the window of valid penalization parameters is relatively narrow, implying that $r_0$ should be chosen with caution.

By contrast, with the exponential penalty $p(x) = \exp(x)$, since $\beta(r)$ in this case can be any polynomial in $r$, the range of admissible values of $r_n$ is much wider.
Performing the same experiment with the exponential penalty (Figure~\ref{fig: toy example result}\subref{subfig: exp-pen-density }, top two rows) shows that it is much less sensitive to the choice of $r_0$, with good distributional convergence as soon as $r_0$ is large enough.

In the bottom two rows, we fix $r_0 = 1$ and assess the qualitative convergence of $\Delta w_{n, r_0}$ to a standard Gaussian for small and large sample sizes.
Overall, performance of the exponential penalty function appears more reliable.

\begin{figure}
    \centering
    \begin{subfigure}{0.33\textwidth}
    
        \includegraphics[width=\linewidth]{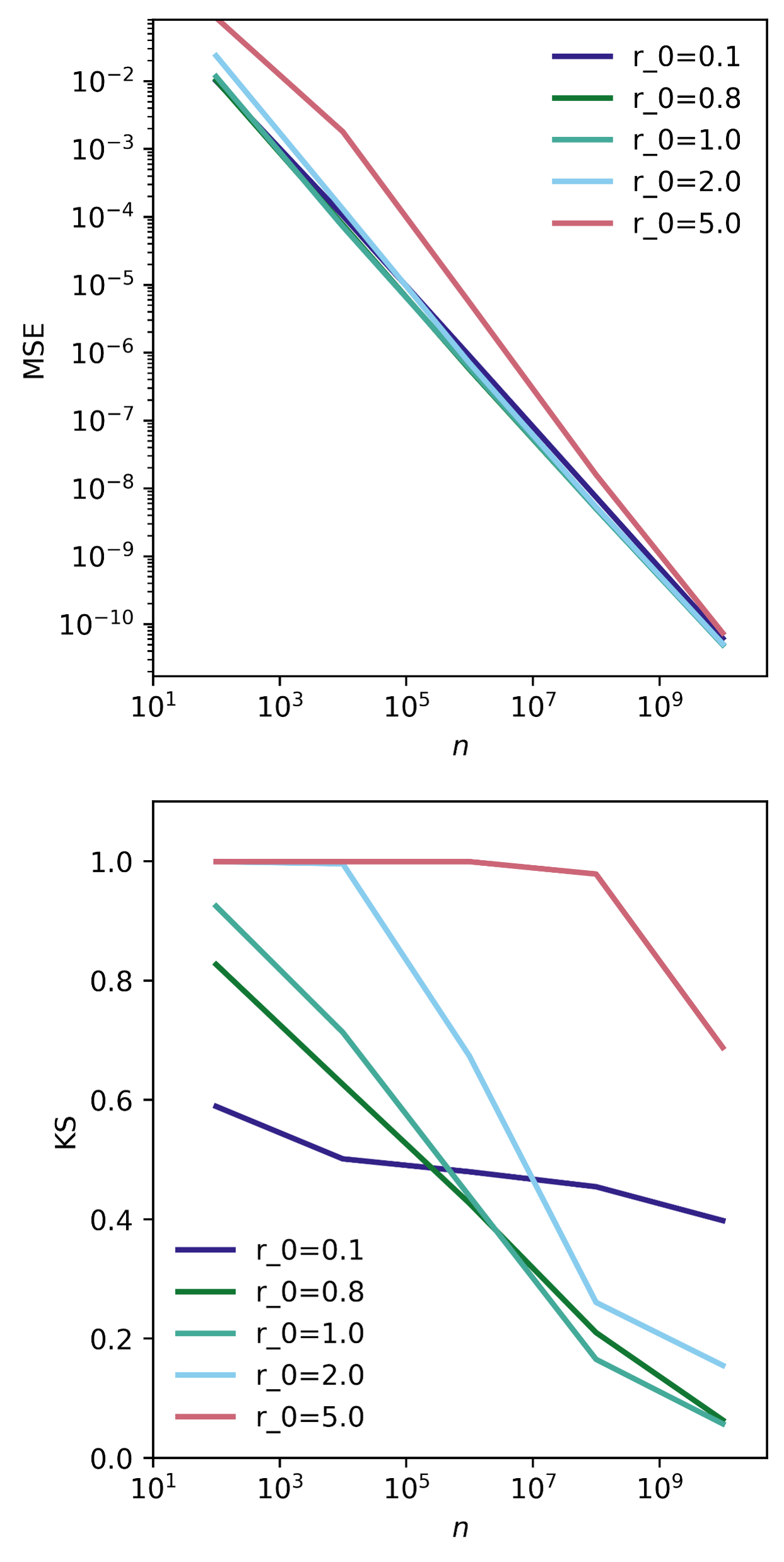}
    \end{subfigure}
    \begin{subfigure}{0.33\textwidth}
        \includegraphics[width=\linewidth]{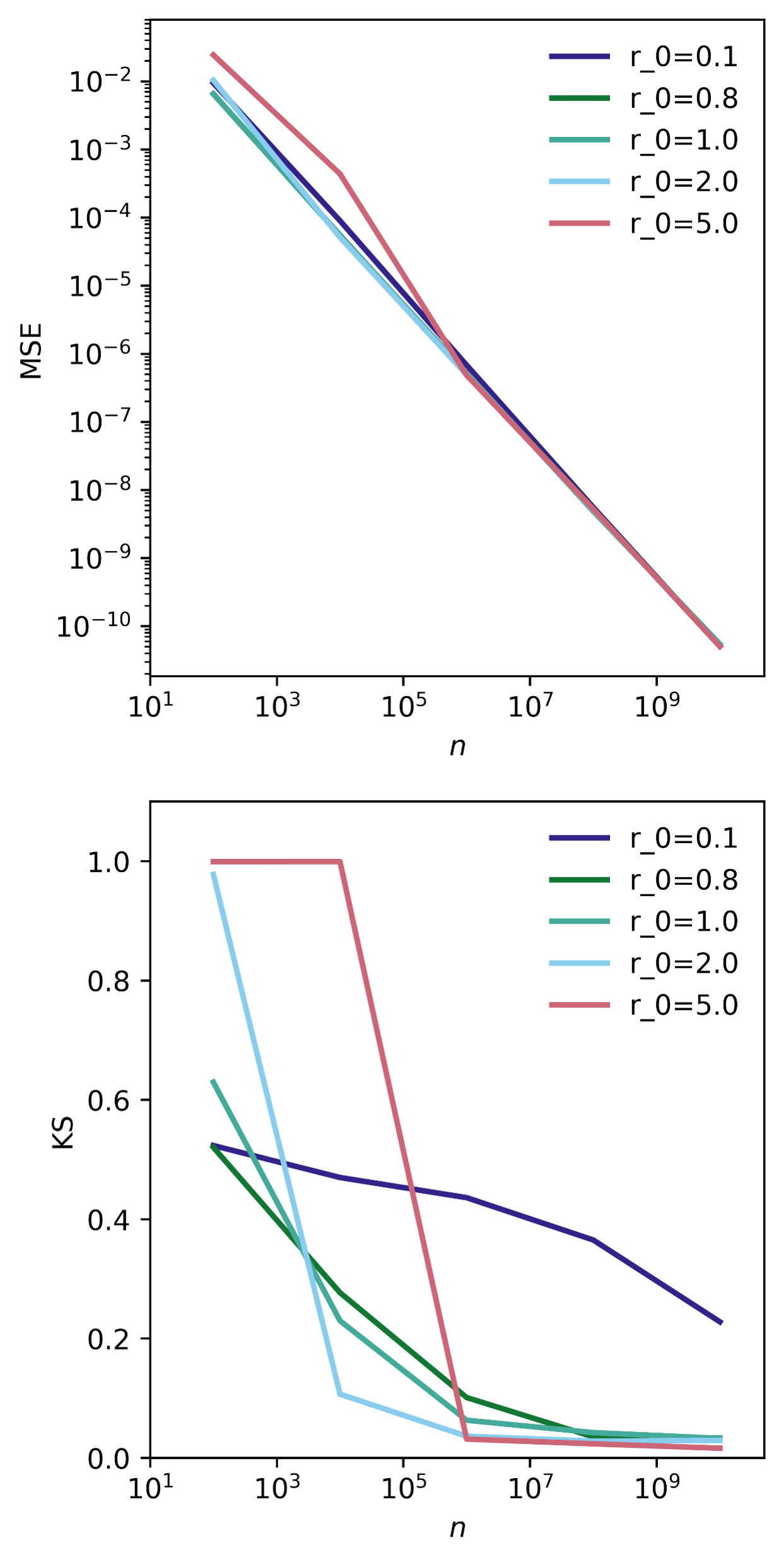}
    \end{subfigure}
    
    \begin{subfigure}{0.45\textwidth}
        \includegraphics[width=\linewidth]{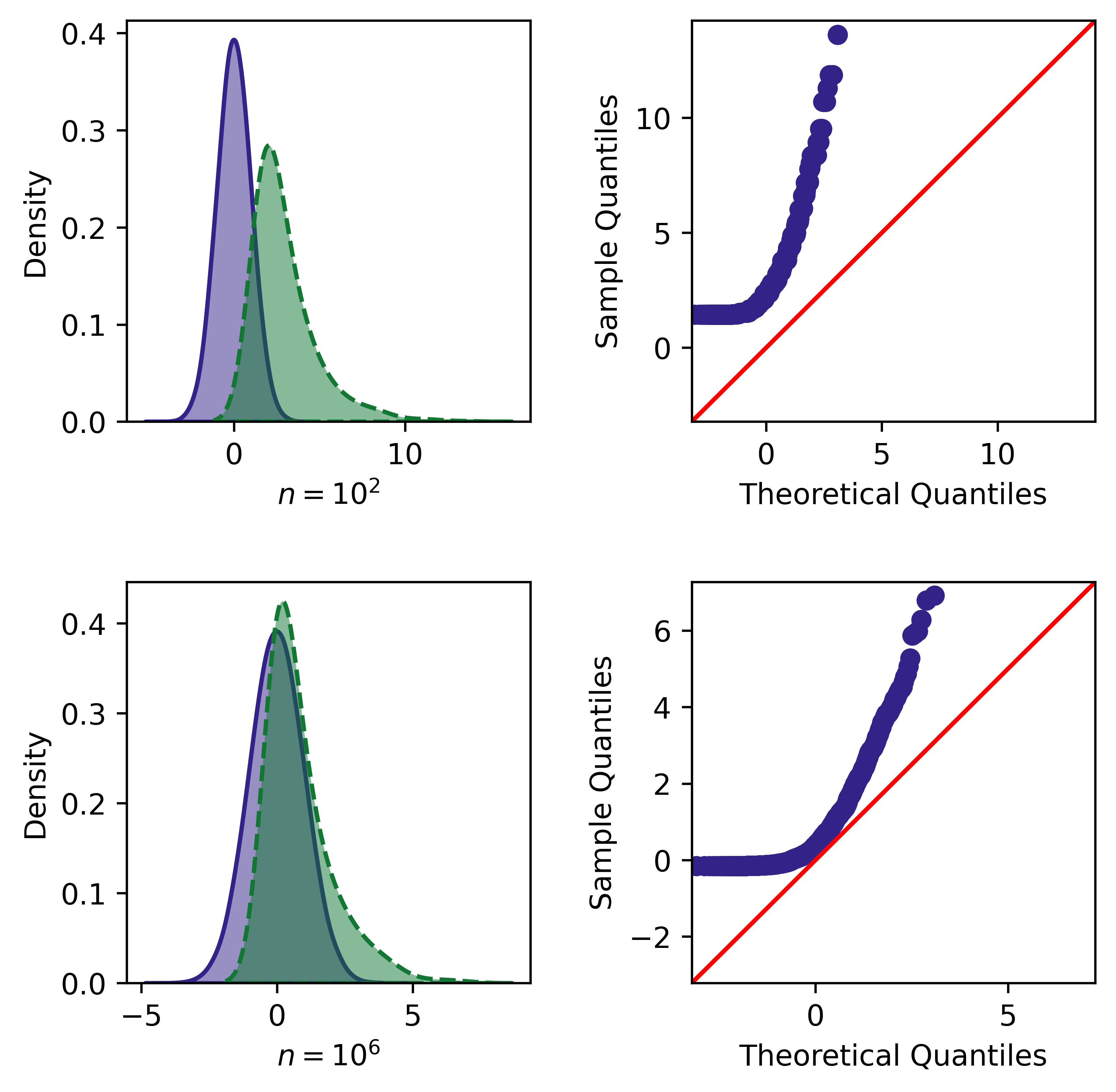}
        \caption{log penalty function }
        \label{subfig: log-pen-density }
    \end{subfigure}
    \begin{subfigure}{0.45\textwidth}
        \includegraphics[width=\linewidth]{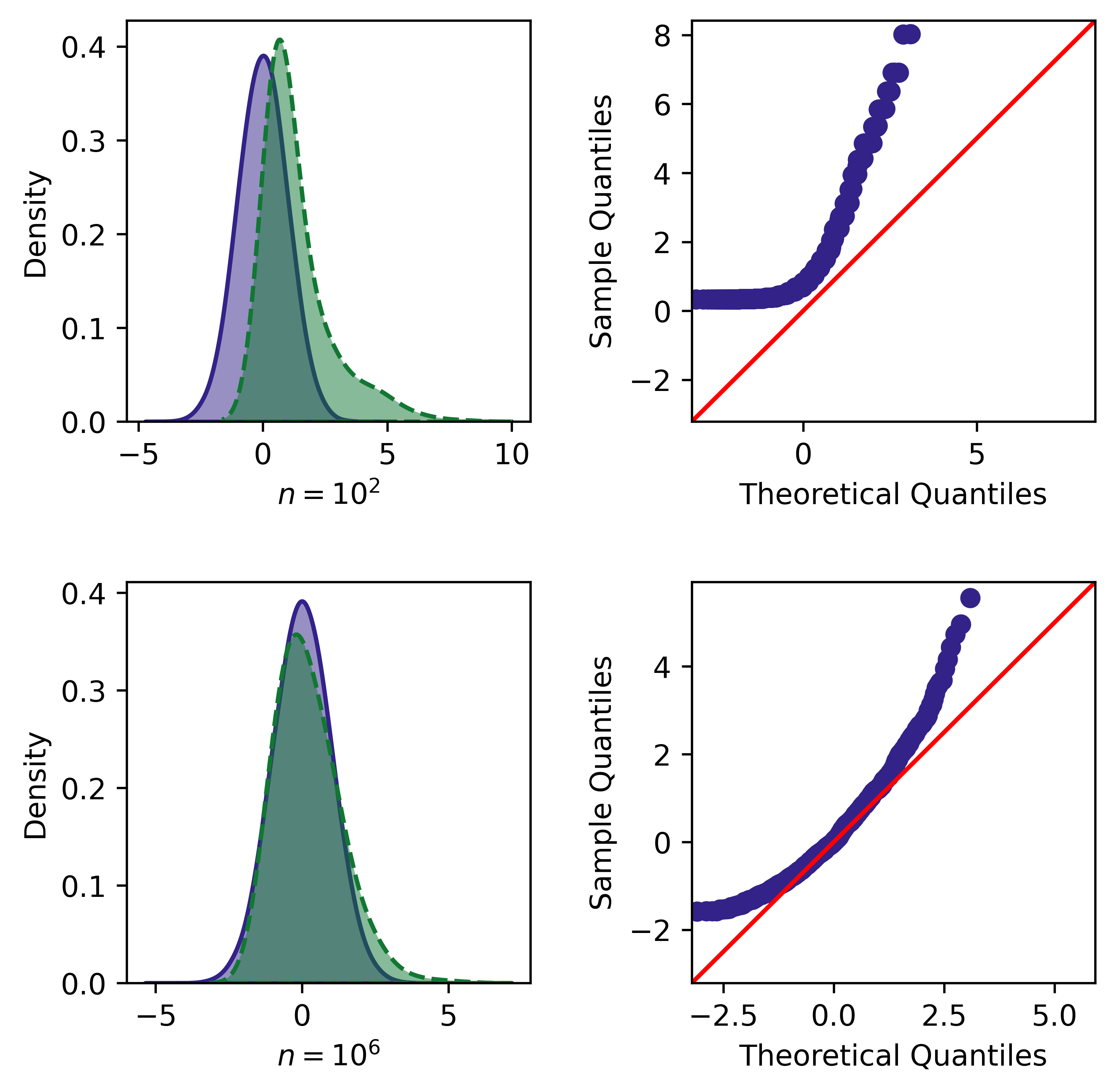}
        \caption{exponential penalty function }
        \label{subfig: exp-pen-density }
    \end{subfigure}
    \caption{\textbf{
    		Comparison between the log and exponential penalty functions.} Subfigures (a), (b) show the same experiment conducted with the logarithmic and exponential penalties, respectively, with 1000 independent replicates in each case. The first row shows plots of the mean-square error  $\mathbb{E}\|\hat \pi_{n, r_0} - \pi^\star\|^2$ as the sample size $n$ and regularization parameter $r_0$ vary. The second shows the Kolmogorov--Smirnov distance between $\Delta w_{n, r_0}$ and the standard Gaussian in the same setting. 
    		In the last two rows, we fix $r_0=1$ and consider a small sample size $n=10^2$ and a large sample size $n=10^6$. The dashed line is a kernel density estimate of the density of $\Delta w_{n, r_0}$   and the solid line is the standard Gaussian density function as a reference. On the right of each density function plot is  the corresponding Q--Q plot with a $45$-degree reference line in red.}
    \label{fig: toy example result}
\end{figure}

\subsection{Large scale simulation}
To illustrate the finite sample behavior of our estimator in high-dimensional settings, we employ our method on the same example as in \cite[Section 5]{KlaTamMun20}. Specifically, we examine an optimal transportation problem between two measures supported on an $L \times L$ equispaced grid over $[0,1]\times[0,1]$, where the transportation cost is proportional to the Euclidean distance between grid points.
We take $L =10$ so that the marginal distributions each have support size $10^2$ and the optimal plan is an element of $\RR^{10^2 \times 10^2}$.

For our first simulation, we draw $\bm{t},\bm{s}$ independently and uniformly from the simplex $\Delta_{L \times L}$, then generate empirical distributions $\bm{t}_n$ and $\bm{s}_n$ using  $n$ i.i.d. samples from the two measures, and then compute our estimator of the optimal plan $\hat \pi_{n, r_0}$ and estimator of the cost $\hat w_{n, r_0}$ using the exponential penalty function with coefficient $r_n=\frac{r_0}{L^4n^{1/3}}$, for $r_0$ a positive constant.

Our theoretical results indicate that the rescaled cost $\Delta w_{n, r_0}$ defined in~\eqref{equ: rescaled cost} converges to a standard Gaussian.
We took $2000$ independent replicates with various samples sizes and choices of $r_0$.
The results are plotted in Figure~\ref{fig:large simulation}.

\begin{figure}
    \centering
    \includegraphics[width=0.9\linewidth]{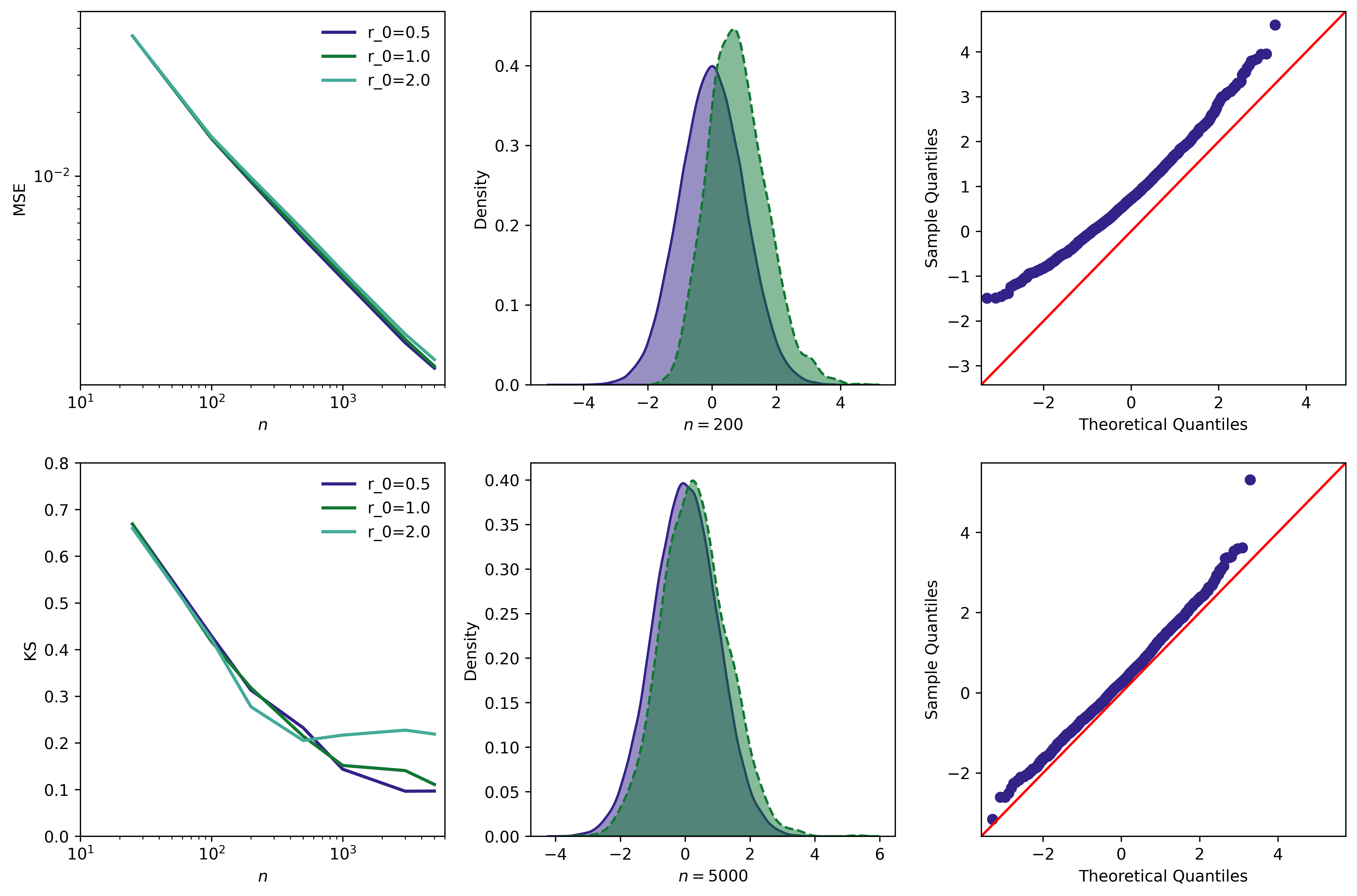}
    \caption{\textbf{Large scale simulation:} Experiments of optimal transportation between two independent Dirichlet distributions on equi-spaced $10\times10$ grid ($2000$ independent replicas in each case). With varies quantities of regularization parameter $r_0$ and sample size $n$, the left two plots shows the Mean Squared Error of the estimated plan ($\mathbb{E}\|\hat \pi_{n, r_0} - \pi^\star\|^2$), and the Kolmogorov–-Smirnov distance between $\Delta w_{n, r_0}$ and the standard Gaussian. The middle two plots show the density comparison between the standard Gaussian (solid lines) and the kernel density estimation of  $\Delta w_{n, r_0}$ (dashed lines) with $r_0=1$ and $n=200$,$5000$ respectively, facilitated with the Q--Q plot on their right.}
    \label{fig:large simulation}
\end{figure}

We also consider the case where $\bm{t} = \bm{s}$.
As discussed in Corollary~\ref{corollary: OT fast convergence}, with the exponential penalty function, we anticipate that $ \sqrt{n}(\hat{w}_{n,r_0}-w^\star)$ will converge in probability to $0$.
Since higher-order asymptotics in this case lie beyond the scope of this paper, we simply visualize the speed of convergence. In this experiment, we take $r_n=\frac{r_0}{L^4n^{1/4}}$ with $r_0=1$. 
We took $2000$ independent replicas of the empirical cost $\hat{w}_{n,r_0}$ and plot a kernel density estimate of $ \sqrt{n}(\hat{w}_{n,r_0}-w^\star)$ when $n = 5000$ and $n=10^9$.
We also plot the Mean Squared Error of $\hat{w}_{n,r_0}$ as a function of $n$ on a log--log plot and estimate the rate of decay via linear regression. The results are plotted in Figure \ref{fig:degenerate}. The experiment shows that $\mathbb{E}(\hat{w}_{n,1} - w^\star)^2=o(n^{-1})$ which are consistent with the finding that $\hat{w}_{n,r_0} - w^\star = o_p(n^{-1/2})$. 

 \begin{figure}
    \centering
    \includegraphics[width=0.9\linewidth]{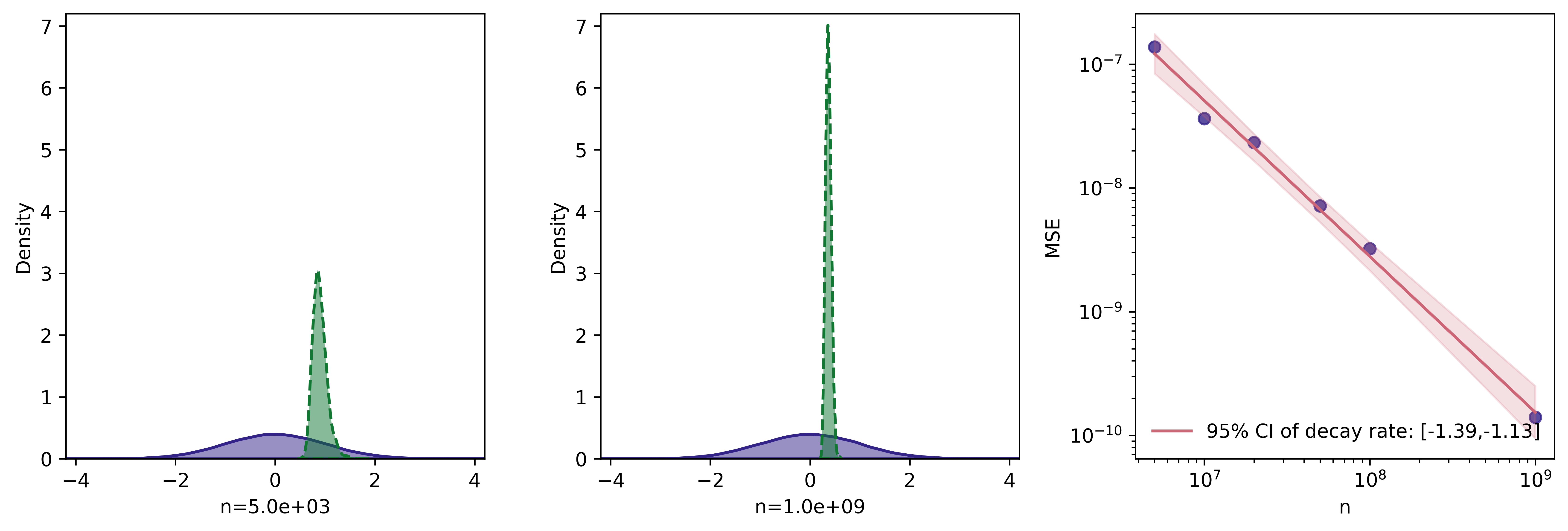}
    \caption{\textbf{Fast convergence with exponential penalty function in the $\bm{t}=\bm{s}$ case:} The left  two plots show the density comparison between the standard Gaussian (solid lines) and  the kernel density estimation of $\sqrt{n}(\hat{w}_{n,1} - w^\star) $ (dashed lines)  with $n=5000$ and $n=10^9$ respectively. The right log--log plot shows the rate of decay of $\mathbb{E}(\hat{w}_{n,1} - w^\star)^2$ with respect to $n$ in the large $n$ regime.  }
    \label{fig:degenerate}
\end{figure}

\section{Numerical Examples}\label{sec:numerical}
In this section, we apply our method to two real-data examples.
The first is an analysis of the bike reallocation problem for the Citi-Bike program in New York City.
In the second, we revisit the protein colocalization analysis of~\cite{KlaTamMun20} and illustrate the benefits of our method in comparison to entropic regularization.
In both examples, we adopt the log-penalty function due to its computational advantages: (1) it is widely supported by convex optimization solvers, and (2) it offers better numerical stability compared to the exponential penalty function.

\subsection{Citi-Bike reallocation problem}

Citi-Bike is a bike-sharing program in New York City which allows residents and tourists to rent bikes from various docking stations across the city.  The Citi-Bike program faces a persistent challenge with bike imbalance, where some docking stations experience high demand while others have surplus bikes. This imbalance, common in busy urban areas, means that in peak locations, riders may struggle to find available bikes or empty docks for returns. To alleviate this problem, Citi-Bike implements daily rebalancing by transporting bikes from stations with a daily surplus to stations with daily deficits. Effective rebalancing improves user satisfaction and increases the system's overall efficiency, making the bike-sharing service more reliable and convenient for daily commuters and tourists alike.

We formulate the bike rebalancing problem as a linear program, and use open-sourced daily trip data provided by Citi-Bike~\citep{citibike_tripdata}  to estimate the average daily need for rebalancing across the Citi-Bike system.
We take daily trip data for $n = 84$ weekdays between June and September 2023 in Manhattan, the New York City borough where Citi-Bike primarily  operates.
We record the net bike flow each day for all $N = 678$ stations in a vector $\bm{d}^{(i)}\in \mathbb{Z}^N$, for $i = 1, \dots, n$.
A positive entry in $\bm{d}^{(i)}$ means that there were more bikes docking at that station than bikes being taken away from the station during the  $i$th day, and vice versa.

We assume that the vectors $\bm{d}^{(i)}$ are i.i.d.\ samples of an underlying random variable, so that the central limit theorem gives $\sqrt{n}(\bar{\bm{d}}_n-\bm{d})\overset{D}\to \mathcal{N}(\bm{0},\bm{\Sigma})$ for some deterministic mean $\bm{d}$ and covariance $\bm{\Sigma}$, where $\bar{\bm{d}}_n = \tfrac 1n \sum_{i=1}^n \bm{d}^{(i)}$.
The problem of rebalancing the bikes among stations while reducing the total cost of rebalancing can be formulated as a linear program:
\begin{equation}
    \bm{\pi}^\star=\argmin_{\bm{\pi}\in\mathbb{R}^{N\times N}} \langle \bm{c},\bm{\pi}\rangle,\quad \text{s.t.} \ \pi_{i,j}\geq 0, \ \sum_{j=1}^{N} \pi_{i,j}-\pi_{j,i}  = {d}_i,
\end{equation}
where $\pi_{i,j}$ is the number of bikes that are transported from station $i$ to station $j$ under a bike-imbalance vector $\bm{d}$ and where $c_{i, j}$ denotes the cost of transporting bikes from station $i$ to station $j$.
Understanding the optimal rebalancing strategy $\bm{\pi}^\star$ can provide the company with a strategy for hiring and assigning workers and trucks. 

We use Theorem~\ref{theorem: debiasing} with the log-penalty function to obtain an estimator $\hat {\bm{\pi}}_n$ of the optimal rebalancing strategy $\bm{\pi}^\star$, and we use $B = 100$ bootstrap samples to generate entrywise  $95\%$ confidence intervals as in Corollary~\ref{cor:bootstrap_quantiles}.
We plot our results on a map of Manhattan in Figure~\ref{fig: citi-bike}.
Purple lines on the map indicate pairs of stations for which the corresponding entry of $\hat {\bm{\pi}}_n$ is at least $1$ and for which the corresponding confidence interval excludes zero.

Our results recover observed patterns in bike rebalancing obtained from the monthly operational report provided by Citi-Bike.
Citi-Bike provides incentives for bike rebalancing at certain locations, which are revealed by our estimator.
We also found significant need for rebalancing near Central Park, especially around Columbus Circle,  a major transportation hub where tourists usually get off from subways to visit this famous park. 

\begin{figure}
    \centering
    \includegraphics[width=\linewidth]{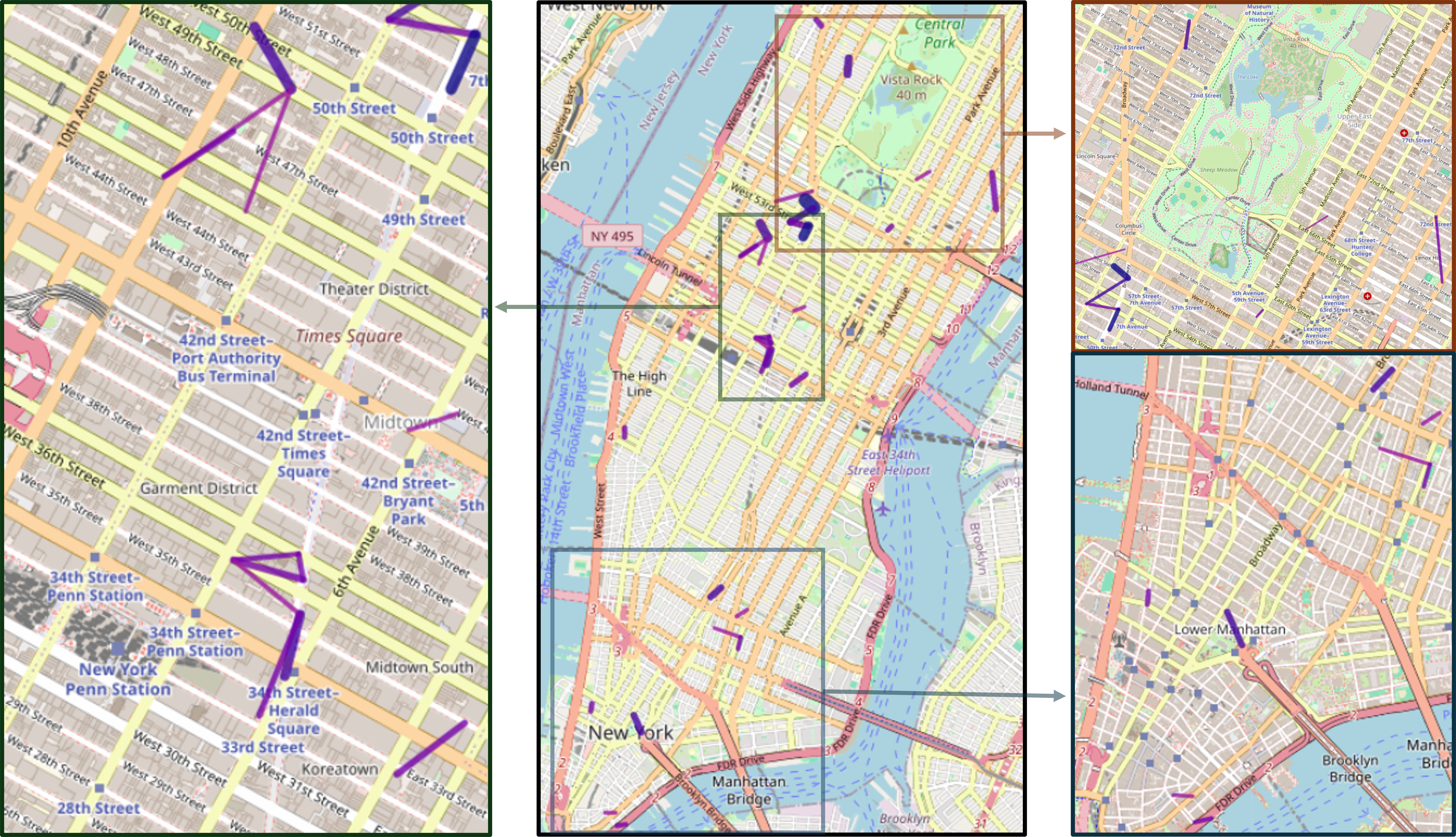}
    \caption{Empirical bike transportation plan among all stations in Manhattan. The estimated quantities of bikes to be transported between station pairs are shown in purple lines. The two ends of a line are the locations of two stations. A thicker and darker line indicates that more bikes need to be transported between the two stations.  Middle: overall plan. Left: Details of plan near Penn Station and Port Authority Bus Terminal.  Upper right: Details of plan around Central Park. Lower right:  Details of plan around Lower Manhattan and East Village. }
    \label{fig: citi-bike}
\end{figure}

\subsection{Colocalization Analysis}
In this section, we compare the performance of our estimator with that of the entropic penalization approach~\eqref{eq:reg_ot_problem} on the same data analyzed by~\cite{KlaTamMun20}, a protein colocalization study consisting of fluorescence microscopy images of cells.
In this data set, pixel intensities reflect the location and amount of proteins in each image.
Analyzing the spatial correlations of different proteins gives insight into cellular behaviors.

\cite{KlaTamMun20} introduce a colocalization measure to quantify the spatial correlation of two proteins. The normalized gray-scale fluorescence microscopy images can be viewed as discrete probability distributions with support on the equidistant grid of pixels indicating concentration of certain proteins. Let $\bm{t}$ and $\bm{s}$ be the distributions of two different proteins. The optimal transport plan between these  two distributions with cost proportional to the Euclidean distance between pixels is denoted by { $ \pi^\star$.} The colocalization measure, as a function of distance $\xi$, quantifies the fraction of proteins which are ``colocalized''---i.e., which are transported only a short distance under the optimal plan. Specifically, the colocalization measure is defined as {
\begin{equation}
     \mathrm{Col}^\star := \mathrm{Col}({\pi}^{\star}) (\xi) = \sum_{i,j=1}^{N} \pi^{\star}_{ij} \mathbbm{1}\{c_{ij} \leq \xi\},
 \end{equation}}
where $N=N_xN_y$ is the number of total pixels on the $(N_y\times N_x)$-size images and  $\bm{c}\in \mathbb{R}^{N\times N}$ stores the Euclidean distance between the pair of pixels on the $(N_y\times N_x)$ grid.
 More generally, for any joint distribution  $\pi := \pi (\bm u, \bm v)$ with marginals $\bm u, \bm v$, we have the colocalization measure for the plan $\bm{\pi}(\bm{u}, \bm{v})$ as
 \begin{equation*}
     \mathrm{Col}(\bm{\pi}(\bm{u}, \bm{v}))(\xi) := \sum_{i,j=1}^{N} \pi_{ij} \mathbbm{1}\{c_{ij} \leq \xi\}.
 \end{equation*}

Computing the colocalization measure directly is computationally hard in practice. For example, even $128\times 128$-pixel images will generate a transportation plan with $2^{28}$ variables. To reduce the storage requirement and to save computational time, \cite{KlaTamMun20} propose subsampling from $\bm{t}$ and $\bm{s}$ to obtain empirical measures $\bm{t}_n$ and $\bm{s}_n$ with $n< N$,  and using these to estimate the true colocalization measure. Although both empirical measures lie in the probability simplex in $\RR^N$, the sizes of their supports are at most $n$ after the subsampling procedure. Due to the nonnegativity constraint on the transportation plan, mass can only be transported between the supports of the two measures. This reduces the computation of the optimal plan to the space of $\RR ^{ | \text{supp} (\bm{t}_n)|\times |\text{supp}(\bm{s}_n )| }$.







{

\cite{KlaTamMun20} use the entropic regularization discussed in Section \ref{sec: entropic} to compute the empirical regularized plan $ { \pi}_{\lambda, n} =: {\pi}_{\lambda,n}(\bm{t}_n, \bm{s}_n)$, which they show enjoys a central limit theorem centered at the population-level regularized plan $\pi^{\star}_{\lambda} : = \pi_\lambda({\bm t, \bm s})$. Consequently, their proposed  estimator 
$\widehat{\mathrm{RCol}}_{n,n} := \mathrm{Col}(\bm{\pi}_{\lambda,n})(\xi)$ enjoys a central limit theorem when centered at 
$\mathrm{{RCol^{\star}_{\lambda}}}:=\mathrm{Col}({\pi}_{\lambda}^\star)(\xi)$.

}



However, they do not compare the regularized colocalization measure $\mathrm{{RCol^{\star}_{\lambda}}} $ with the true colocalization measure $\mathrm{{Col}^{\star}}$ to assess the effect of the bias introduced by entropic regularization.
We apply the method developed in this paper to compute the debiased estimator $\hat{\bm{\pi}}_n$ using the log-penalty function, and subsequently obtain the corresponding colocalization estimator. To distinguish our approach from the entropic regularization method proposed by \cite{KlaTamMun20}, we refer to it as the penalization method and denote the resulting colocalization estimator as $\widehat{\mathrm{PCol}}_{n,n} := \mathrm{Col}(\hat{\bm{\pi}}_n)(\xi)$. In both methods, empirical transport plans are restricted to the positive orthant, effectively reducing the optimization domain to $\RR^{|\text{supp}(\bm{t}_n)| \times |\text{supp}(\bm{s}_n)|}$.

For a $128 \times 128$-pixel image, we can compute (at considerable computational cost) the ``population-level'' colocalization $\mathrm{Col}$.
We can therefore compare the estimators $\widehat{\mathrm{PCol}}_{n,n}$ and $\widehat{\mathrm{RCol}}_{n,n}$ with $\mathrm{Col}^\star$. 

To make a fair comparison, we use the same  2-colour-STED images of ATP Synthase and MIC60 from \cite{tameling2021simluated} (shown in figure~\ref{fig: STED image}) and apply the same parameter settings ($n = 2000$ samples with regularization parameter $\lambda = 2$) as \cite{KlaTamMun20}.
We compute $\widehat{\mathrm{PCol}}_{n,n}$ and $\widehat{\mathrm{RCol}}_{n,n}$ as well as $95\%$ confidence bands for each estimator with $B = 100$ bootstrap replicates.
\begin{figure}
\centering
    \begin{subfigure}[c]{0.3\textwidth}
        \centering
        \includegraphics[width=\textwidth]{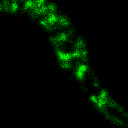}
    \end{subfigure}
    \begin{subfigure}[c]{0.3\textwidth}
        \centering
        \includegraphics[width=\textwidth]{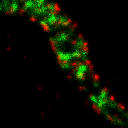}
    \end{subfigure}
     \begin{subfigure}[c]{0.3\textwidth}
        \centering
        \includegraphics[width=\textwidth]{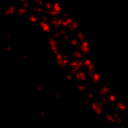}
    \end{subfigure}
    \caption{Zoomed in $128\times 128$ pixels of STED images. Pixel size=$15nm$. Left: ATP Synthase in green channel; Right: MIC60 in red channel; Middle: overlay of the green and the red channel.}
    \label{fig: STED image}
\end{figure}
Since $\mathrm{Col}({\pi})$, viewed as mapping from plans to the space of \textit{c\`adl\`ag} functions, is linear and Lipschitz, Theorem~\ref{theorem: debiasing} implies that
\begin{equation}
    \sqrt{n}(\widehat{\mathrm{PCol}}_{n,n}-\mathrm{Col}^{\star})\overset{D}\to \mathrm{Col}(\bm{M}^\star\bm{G}),
\end{equation}
where $\mathbb{G}$ is the asymptotic limit of $\sqrt n (\bm{t}_n - \bm t, \bm{s}_n - \bm s)$. 
We construct uniform confidence bands by letting $u_{1-\alpha}$ be the $1-\alpha$ quantile of $\|\mathrm{Col}(\bm{M}^\star\bm{G})\|_{\infty}$ and defining
\begin{equation}
    \mathcal{I}_n := \left[ -\frac{u_{1-\alpha}}{\sqrt{n}} + \widehat{\mathrm{PCol}}_{n,n}, \frac{u_{1-\alpha}}{\sqrt{n}} + \widehat{\mathrm{PCol}}_{n,n} \right],
\end{equation}
which is an asymptotic $1-\alpha$ confidence interval for $\mathrm{Col}$.
By the results of Section~\ref{sec: bootstrap}, we can estimate $u_{1-\alpha}$ by bootstrapping $\bm{t}_n$and $ \bm{s}_n$.
We plot our results in Figure~\ref{fig: compare reg and pen}.

\begin{figure}
    \centering
    \begin{subfigure}[c]{0.45\textwidth}
        \centering
        \includegraphics[width=\textwidth]{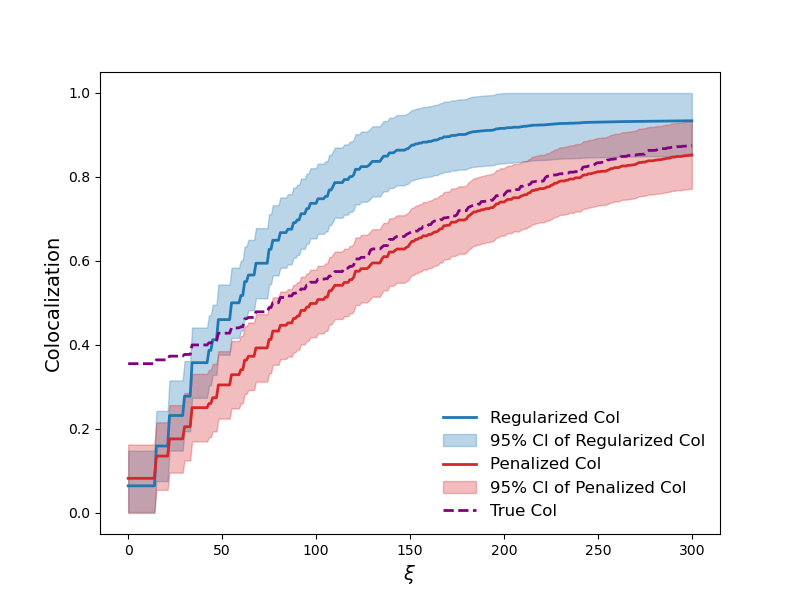}
        \caption{$95\%$ confidence intervals for true colocalization function, computed with regularized Col estimator ($\widehat{\mathrm{RCol}}_{n,n}$) and penalized Col estimator ($\widehat{\mathrm{PCol}}_{n,n}$).}
        \label{fig: compare reg and pen}
    \end{subfigure}
    \begin{subfigure}[c]{0.45\textwidth}
        \centering
        \includegraphics[width=\textwidth]{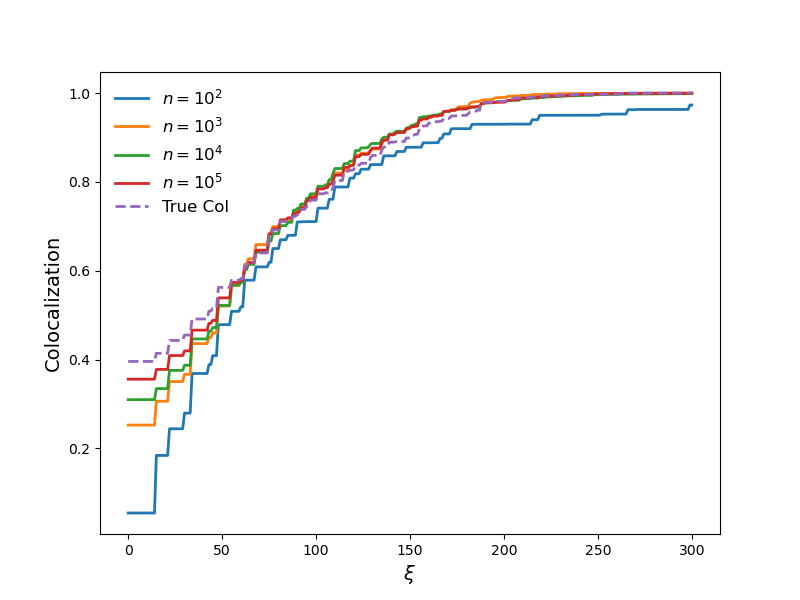}
        \caption{Estimated colocalization $\widehat{\mathrm{PCol}}_{n,n}$ with varying number of samples.}
        \label{fig: colocalzation bias}
    \end{subfigure}
    \caption{Colocalization estimation with subsampled images}
\end{figure}
We observe that our approach yields an estimator that is close to  $\mathrm{Col}^{\star}$ for all sufficiently large $\xi$, and that our confidence set captures $\mathrm{Col}^{\star}$ for a large range of $\xi$.
By contrast, $\widehat{\mathrm{RCol}}_{n,n}$ substantially deviates from $\mathrm{Col}^{\star}$, and the corresponding confidence set is not accurate.
This observation is consistent with our discussion in Section~\ref{sec: entropic}: using entropic optimal transport yields substantially biased estimates of the true optimal plan.
Neither method succeeds in estimating $\mathrm{Col}^{\star}(\xi)$ when $\xi$ is small.
This is an unavoidable feature of the subsampling method: the original distributions $\bm t$ and $\bm s$ have support of size $16384$, but the subsamples have much smaller support, of size at most $2000$.
This extremely sparse sampling means that $\bm t_n$ and $\bm s_n$ fail to capture the small-scale structure of the optimal plan.
We show in Figure~\ref{fig: colocalzation bias} an example on a smaller image (of size $32 \times 32$), where it is clear that the errors present for small $\xi$ substantially decrease as the sample size increases.

\section{Conclusion}
In this work, we propose and analyze a new inference technique for optimal solutions of optimal transport problems and other linear programs.
Unlike the widely applied entropic regularization, our approach yields asymptotically unbiased estimators, centered at the true population-level quantities.
Our main insight is to construct a penalized linear program whose bias is linear in the penalty parameter, which allows us to apply the Richardson extrapolation technique to eliminate the bias.
In future work, it would be valuable to investigate whether similar techniques could be used to develop new inference methods for non-parametric settings, such as for the solution to optimal transport problems with continuous marginals~\citep{GonLouNil22,ManBalNil23}.
\section*{Acknowledgments}

Florentina Bunea was supported in part by NSF--DMS 2210563. Jonathan Niles-Weed was supported in part by NSF--DMS 2210583 and 2339829.

\appendix

\section{Preliminaries}
We first introduce some general preliminary results that will be applied in the proof of our main theorems. 
\subsection{Random linear programs}
\label{sec:preliminary}
Linear programs like~\eqref{primal} have polytope feasible sets, whose vertices are determined by the feasible hyperplane  $\bm{Ax}=\bm{b}$. If solutions exist, the linear program must have a solution at one of the vertices of the feasible polytope.  
 For any subset $I \subseteq \{1, \dots, m\}$, we denote by $\bm{A}_I$ the $k \times |I|$ submatrix of $\bm{A}$ formed by taking the columns of $\bm{A}$ corresponding to the elements of $I$.
Analogously, for $\bm{x} \in \RR^m$, we write $\bm{x}_I$ for the vector of length $|I|$ consisting of the coordinates of $\bm{x}$ corresponding to $I$.

\begin{definition}
	A set $I \subseteq [m]$ is a \emph{basis} if
	\begin{equation}
	|I|=k, \qquad \operatorname{rank}(\bm{A}_I)=k
	\label{Intro:index_set}
	\end{equation}
\end{definition}
 Given a basis $I$, we can define the \emph{basic solution} $\bm{x}(I;\bm{b})$ to be the vector $\bx$ satisfying 
\begin{equation}\label{basis_def}
	\begin{aligned}
	\bx_I &= \bA_I^{-1} \bb \\
	\bx_{I^C} &= \bz\,.
	\end{aligned}
\end{equation}
If this basic solution $\bm{x}(I;\bm{b})$ is also feasible to~\eqref{primal}, i.e., $\bm{x}_I$ is nonnegative, we say that $\bm{x}(I;\bm{b})$ is a \textit{feasible basic solution}. We define the optimal set of bases for~\eqref{primal} as $\mathcal{I}^\star (\bm{b})$, which contains all the bases corresponding to optimal solutions.

We say that the linear program has a \textit{unique non-generate solution} if $\bm{x}^\star$ is the unique solution to the linear program and $\operatorname{supp}(\bm{x}^\star)=k$. In this situation, $\mathcal{I}^\star(\bb)$ has only one element.

Let $I$ be a basis of $\bm{A}$, we show that if $\bm{x}$ is in the hyperplane $\bm{Ax}=\bm{b}_0$, $\bm{x}$ is uniquely determined by its entries on $I^C$.
 \begin{lemma}
 \label{lemma: I_0 is enough}
     Under the condition that the linear program~\eqref{primal} has a unique solution $\bm{x}^\star _{LP}$ with the index set for the zeros entries $I_0\coloneqq\{ i\in\{1,...,m\}|x^\star _{LP,i}=0 \}$, the linear equation for $\bm{x}$:  
     \begin{align*}
         \bm{Ax}&=\bm{b}_0,\\
         x_i&=\alpha_i,\quad i\in I_0 
     \end{align*}
     cannot have multiple solutions, where $\bm{A}$ is the same matrix as shown in the equality constraints of the linear program and $\alpha_i's$ are constants.    

     If the equation has a solution $\bm{x}^\star $, we have $ \bm{x}^\star _{I_0^c}$ as a linear function of $(\bm{b}_0, \bm{\alpha})$ and there exists a constant $C$ depending on the matrix $A$, and $I_0$ such that  $ \|\bm{x}^\star \| \leq C (\max_{i \in I_0} |\alpha_i|+\|\bm{b}_0\|)$.
 \end{lemma}
  \begin{proof}
     Since the linear program~\eqref{primal} has unique solution, according to the properties of solutions to linear programs \citep[Theorem 2.2]{bertsimas1997introduction}, there is a basis $I$ such that   $I^c\subseteq I_0$   and the constraints $x_i\geq 0, i\in I^c$ and $\bm{Ax}=\bm{b}$ are linearly independent. Therefore, the equation below, which has reduced constraints,  has a unique solution. This indicates that the original equation has at most one solution. \begin{align*}
         \bm{Ax}&=\bm{b}_0,\\
         x_i&=\alpha_i,\quad i\in I^c 
     \end{align*}

    The solution to the reduced equation $\bm{x}^\star $ can be written as 
    \begin{align*}
        \bm{x}^\star _{I}&=\bm{A}_{I}^{-1}(\bm{b}_0-\bm{A}_{I^c}\bm{x}^\star _{I^c}),\\
        x_i&=\alpha_i,\quad i\in I^c,\\
    \end{align*}
    where $\bm{A}_{I_c}$ and $\bm{A}_{I}$ are submatrices of $\bm{A}$ taking columns with indices of $I_c$ and $I$ separately.

    Therefore, if the equation in the lemma has a solution $\bm{x}^\star $,
    \begin{equation}
        \|\bm{x}^\star \|\leq (m-k)(\|\bm{A}_{I}^{-1}\bm{A}_{I^c}\|+1)  \max_{i\in I_0} |\alpha_i|+\|\bm{A}_{I}^{-1}\|\|\bm{b}_0\|
    \end{equation}
        
 \end{proof}

 Dual programs are utilized a lot in the discussion of convex programs. The dual linear program for~\eqref{primal} is 
 \begin{equation}
    \max_{\bm{\lambda}\in \mathbb{R}^k} \langle \bm{b}, \bm{\lambda}\rangle,\qquad \textrm{s.t.}\ \bm{c}-\bm{A^\top\lambda}\geq 0. 
    \label{dual}
\end{equation}
The primal and dual linear programs achieve optima at $(\bm{x}^\star (\bm{b}), \bm{\lambda}^\star (\bm{b}))$ if and only if $\exists\  \bm{\eta}\in \mathbb{R}^m$ such that:
\begin{equation}
    \label{optimality condition}
    \bA^\top\bm{\lambda}^\star (\bb)+\bm{\eta}=\bc,\ \bA\bx^\star (\bb)=\bb,\ \bx^\star (\bb)\geq 0,\ \bm{\eta}\geq 0, \ \bx^\star (\bb)^\top\bm{\eta}=0. 
\end{equation}
The \textit{complementary slackness} condition $\bx^\star (\bb)^\top\bm{\eta}=0$ is equivalent to the condition that $\operatorname{supp}(\bx^\star (\bb))\cap\operatorname{supp}(\bm{\eta})=\emptyset$. If it also holds that $\operatorname{supp}(\bx^\star (\bb))\cup\operatorname{supp}(\bm{\eta})=[m]$, we say that the solutions satisfy the \textit{strict complementary slackness} condition.  Linear programs have their optimal solutions at vertices or on optimal faces of the feasible polytope. When the Slater's condition holds for both dual and primal problems, the optimal faces are bounded ~\citep[Corollary 7.1k]{schrijver1998theory} and the strict complementary slackness condition can be achieved at the centroids of the primal and dual optimal faces.

When we consider a perturbed linear program where the vector $\bm{b}$ is replaced by a nearby vector $\bm{b}'$,  according to Proposition 2 in~\citep{LiuBunNil23}, we have the inclusion of the optimal basis set $\mathcal{I}^\star (\bm{b}')\subseteq \mathcal{I}^\star (\bm{b})$ if $\|\bm{b}'-\bm{b}\|\leq \delta(\bm{A},\bm{b})$.  Under the assumption for the linear programs~\ref{assumption on LP} which ensures that the primal problem~\eqref{primal} has a unique solution, the feasible basic solution $\bm{x}^\star (\bm{b'})$ can be expressed as
\begin{equation}
\label{appendix: xb' expression}
    \bm{x}^\star (\bm{b'})=\bm{x}^\star (\bm{b})+\bm{x}(I;\bm{b}'-\bm{b}) \text{ for } I\in \mathcal{I}^\star (\bm{b}').
\end{equation}
In the case where $\bm{x}(I; \bm{b}'-\bm{b})$ is not unique, the optimal solution to the linear program with parameter $\bm{b}'$ is achieved at an optimal face with the points $\bm{x}^\star (\bm{b'})$ as its vertices. 
 
 In the situation where $\bm{b}$ in the linear program~\eqref{primal} is replaced by a random vector $\bm{b}_n$ (such as a plug-in estimator for $\bm{b}$ with distributional limit given in~\eqref{converge bn}), our previous result demonstrates that the random solution shares the same basis as the true solution. Furthermore, the limit distribution of the random solution has a nonlinear relationship with $\mathbb{G}$.
 \begin{theorem}[Theorem 3, \citet{LiuBunNil23}]
 \label{nonGaussion limits}
   Suppose that~\eqref{primal} satisfies~\ref{assumption on LP}.
	If $\bm{b}_n$ satisfies the distributional limit~\eqref{converge bn},  then
	\begin{equation}
		q_n (\bx^\star (\bb_n) - \bx^\star (\bb)) \overset{D}\to\bm{p}^\star _{\bb} (\mathbb{G})\,,
	\end{equation}
	where $\bm{p}^\star _{\bb} (\mathbb{G})$ is the set of optimal solutions to the following linear program:
	\begin{equation}
		\min \langle \bc, \bm{p} \rangle: \bA \bm{p} = \mathbb{G}, \quad \bm{p}_{i} \geq 0 \quad \forall i \notin S(\bx^\star (\bb))\,.
  \end{equation} 
  Here $\bx^\star (\bb_n)$ and $\bx^\star (\bb)$ are the solutions to the linear programs with $\bb_n$ and $\bb$ in the equality constraints separately, and  $S(\bx^\star (\bb))$ is the support of the $\bx^\star (\bb)$.
\end{theorem}

We have provided a $2\times 2$ optimal transport problem in Section~\ref{sec:intro}, Proposition~\ref{prop:2by2solution} as a small example for such nonlinearity. We give a brief proof here. 
\begin{proof}[Proof of Proposition~\ref{prop:2by2solution}]
Since $\bm{t}_n$ and $\bm{s}_n$ represent empirical frequencies, they are naturally nonnegative and satisfy $t_{n,1}+ t_{n,2}= 1$ and $s_{n,1}+ s_{n,2}= 1$. Then, the feasibility of $\pi_n$ can be easily checked by substituting~\eqref{eq:pi_n expression} into the constraints of the program~\eqref{equ: 2by2 example}.

To show the optimality of $\pi_n$, we consider an  arbitrary alternative feasible plan $\pi'$ and their difference $\Delta \pi=\pi'-\pi_n$. 
Since $t_{n,2}-s_{n,2}=-(t_{n,1}-s_{n,1})$, we have
$\min\{\pi_{n,12},\pi_{n,21}\}=0.$
Since both $\pi_n$ and $\pi'$ obey the  constraints in~\eqref{equ: 2by2 example}, $\Delta \pi$ takes the form:
\begin{equation}
    \Delta \pi=\begin{pmatrix}
			-\delta & \delta\\
			\delta & -\delta\,
		\end{pmatrix}\,,
\end{equation}
for $\delta>0$. 

Thus, we have the comparison of transport cost
\begin{equation}
    \pi'_{12}+\pi'_{21}>\pi_{n,12}+\pi_{n,21}.
\end{equation}
This indicates that $\pi_n$ is the only minimal cost transport plan. 

We have the convergence law for the empirical frequency:
\begin{equation}
    \sqrt{n}(\bm{t}_n-\bm{t})\overset{d}\to\frac{1}{2}\mathbb{G},
\end{equation}
where $\mathbb{G}$ represents a standard Gaussian random variable and the same law applies the $\bm{s}_n$.
Therefore, the convergence law~\eqref{eq:pi_n limit} is a direct result of applying the continuous mapping theorem to $\pi_n$.

\end{proof}

\subsection{Properties of convex functions and convex optimization}
Firstly we introduce the key lemma we will apply during the proof of convergence of the penalized program trajectories. Intuitively, this lemma establishes a lower bound on how quickly a locally strongly convex function grows compared to its tangent hyperplane. It guarantees that when moving from a point to any other point in the domain (both within an affine subspace), the function increases by at least a constant times either the distance between points or the squared distance, whichever is smaller. This bound captures both the quadratic growth behavior near the reference point and the potentially faster growth far from it. 
\begin{lemma}[Lower bound on convex function growth]\label{lem:new_norm_bound}
	Let $\mathbb{A}=\{\bm{\alpha}+\bm{w}:\bm{w}\in \mathcal{V}\}\subset\Rset^m$ be an affine subspace of $\mathbb{R}^m$ where $\mathcal{V}$ is a vector space.
	
	Let $g: \Rset^m \to \overline \Rset$ be a lower semi-continuous convex function.
	Assume that $g$ is $C^2$ on $\inter \dom g$ and that 
    $g$ is locally strongly convex on $\mathbb A$.
	Let $K$ be a compact subset of $\relint(\dom g \cap \mathbb A)$.
	There exists a constant $C = C(K) > 0$ such that for any $\bd_1 \in K$ and any $\bd_2 \in \dom g \cap\mathbb A$,
	\begin{equation*}
		g(\bd_2) - g(\bd_1) - \langle \nabla g(\bd_1), \bd_2 - \bd_1 \rangle \geq C \min\{ \|\bd_2 - \bd_1\|, \|\bd_2 - \bd_1\|^2\}\,.
	\end{equation*}
\end{lemma}
\begin{proof}
	Define $s(\bd, \bd') := g(\bd') - g(\bd) - \langle \nabla g(\bd), \bd' - \bd \rangle$.
	The locally strong convexity of $g$ within space $\mathbb{A}$ implies that $s$ is strictly positive for all $\bd,\bd'\in \mathbb{A}$ and $\bd \neq \bd'$.
	We first show that for $C_0 > 0$ sufficiently small, the set $\{(\bd_1, \bd_2) \in K \times (\dom g \cap\mathbb A): s(\bd_1, \bd_2) \leq C_0 \|\bd_1 - \bd_2\|\}$ is compact.
	Indeed, if not, we can find a sequence $(\bd_{1,n}, \bd_{2, n})$ with $\|\bd_{2,n}\| \to \infty$ such that
	\begin{equation*}
		\lim_{n \to \infty} \frac{s(\bd_{1,n}, \bd_{2,n})  }{\|\bd_{1,n} - \bd_{2,n}\|} = 0\,.
	\end{equation*}
	The convexity of $g$ implies that for any $\bd \in \dom g\cap\mathbb A$, $\bv \in \mathcal{V}$ and $t \geq 1$
	\begin{equation*}
		\frac{s(\bd, \bd + t \bv)}{t} = \frac{g(\bd + t \bv) - g(\bd)}{t} - \langle \nabla g(\bd), \bv \rangle \geq g(\bd + \bv) - g(\bd) - \langle \nabla g(\bd), \bv \rangle = s(\bd, \bd + \bv)\,.
	\end{equation*}
	Hence
	\begin{equation}\label{eq:sandwich}
		0 = \lim_{n \to \infty} \frac{s(\bd_{1,n}, \bd_{2,n})  }{\|\bd_{1,n} - \bd_{2,n}\|}  \geq \limsup_{n \to \infty} s(\bd_{1, n}, \bd_{1, n} + (\bd_{2, n}-\bd_{1,n})/\|\bd_{1,n} - \bd_{2,n}\|)\,.
	\end{equation}
	Passing to a subsequence, we may assume that $\bd_{1,n} \to \bd_1 \in K$ and $(\bd_{2, n}-\bd_{1,n})/\|\bd_{1,n} - \bd_{2,n}\| \to \bv$ with $\|\bv\| = 1$.
	Since $\bd_1 \in K \subseteq \relint(\dom g \cap \mathbb A)$, both $g$ and $\nabla g$ are continuous at $\bd_1$.
	The lower semi-continuity of $g$ then implies
	\begin{equation*}
		\limsup_{n \to \infty} s(\bd_{1, n}, \bd_{1, n} + (\bd_{2, n}-\bd_{1,n})/\|\bd_{1,n} - \bd_{2,n}\|) \geq s(\bd_1, \bd_1 + \bv)\,,
	\end{equation*}
	which contradicts~\eqref{eq:sandwich}.
	
	Let $C_0$ be such that $S := \{(\bd_1, \bd_2) \in K \times (\dom g \cap \mathbb A): g(\bd_2) - g(\bd_1) - \langle \nabla g(\bd_1), \bd_2 - \bd_1 \rangle \leq C_0 \|\bd_1 - \bd_2\|\}$ is compact.
	Since $g$ is $C^2$ with $\inf_{\bm w \in \mathcal V, \|\bm w\|^2 = 1}\langle \bm w, (\nabla^2 g) \bm w\rangle > 0$ on $\dom g \cap \mathbb A$, there exists $\lambda > 0$ such that $g$ is $\lambda$-strongly convex on the convex hull of $S$.
	In particular, for $(\bd_1, \bd_2) \in S$, it holds
	\begin{equation*}
		\frac \lambda 2 \|\bd_1 - \bd_2\|^2 \leq g(\bd_2) - g(\bd_1) - \langle \nabla g(\bd_1), \bd_2 - \bd_1 \rangle\,.
	\end{equation*}
	
	Therefore either $g(\bd_2) - g(\bd_1) - \langle \nabla g(\bd_1), \bd_2 - \bd_1 \rangle \geq C_0 \|\bd_1 - \bd_2\|$ or $(\bd_1, \bd_2) \in S$, so $g(\bd_2) - g(\bd_1) - \langle \nabla g(\bd_1), \bd_2 - \bd_1 \rangle \geq \tfrac{\lambda}{2} \|\bd_1 - \bd_2\|^2$.
	Choosing $C = \min\{C_0, \lambda/2\}$ proves the claim.
\end{proof}
Below is another lemma related to a barrier-type convex optimization problem. This lemma establishes an upper bound on the gradient magnitudes at the optimal solution when minimizing a barrier function over a compact polytope.  This lemma indicates that if $\mathcal{P}^n$ is a sequence of shrinking polytopes such that $v^n_{max}\ll1$ and $\frac{v^n_{max}}{v^n_{min}}=\Theta (1)$, then with $f$ satisfying conditions in the lemma, we have $\max_{i\in I'}|f'(\bm{x_{0,i}}+\bm{x}^\star _i)|=O(f'(v^n_{max}))$. 
\begin{lemma}
    \label{lemma for 2.1}
    Let  polytope $\mathcal{P}$ in $\mathbb{R}^{m}$ be defined as:
    \begin{equation}
        \mathcal{P}=\{\bm{x}\in {\mathbb{R}^{m}}|\bm{A}\bm{x}=\bm{b},\ \bm{x}_i\geq 0 \text{ for } i\in I\},
    \end{equation}
    where the $\bm{A}\in \mathbb{R}^{k\times m}$ has full row rank and $I $ is a index set.  Suppose that $\mathcal{P}$ is compact with a centroid $\bm{x}_c$.
    
    Let $f(x): (0,+\infty)\rightarrow\mathbb{R}$ be  a convex function in $C^1$ that has the properties: 1)$\lim_{x\downarrow 0^+}f'(x)=-\infty$ 2) $\lim_{x\downarrow 0^+}-xf'(x)<\infty$ and 3) there exits $\delta>0$ such that $xf'(x)$ decreasing in $(0,\delta)$.

    Consider the convex optimization problem:
      \begin{equation}
    \label{equ: minimize in poly}
        \bm{x}^\star =\argmin_{\bm{x}\in \mathcal{P}} \sum_{i=1}^mf(\bm{x}_0+\bm{x})_i,
    \end{equation}
    where $\bm{x}_0\in \bm{R}_{\geq 0}^m$ and $\text{supp }\bm{x}_0=I^c $.

    Suppose that $m=|\cup_{\bm{x}\in \mathcal{P}} \text{supp} (\bm{x}+\bm{x}_0)|$. Let $v_{max}=\max_{\bm{x}\in \mathcal{P}}|\bm{x}|_{\infty}$  and  $v_{min}=\min_{i\in[m]}x_{c,i}$.  Let $x_{max}=\max_{i\in I^C} x_{0,i}$ and $x_{min}=\min_{i\in I^C} x_{0,i}$.
    
    If $v_{max}<\delta$ and  $ x_{min}>2v_{max}$,
    we have 
    $$\max_{i\in [m]}|f'(\bm{x_{0,i}}+\bm{x}^\star _i)|\leq  m \frac{v_{max}}{v_{min}}( |f' (v_{max})|+2(|f'(x_{max}+v_{max})|\vee|f'(x_{min}-v_{max})|)).$$
\end{lemma}
\begin{proof}
    Since every point on the boundary of the polytope $\mathcal{P}$ has at least one zero entry and $f'(x)$ diverges to $-\infty $ at zero, the first order derivative of the cost function in~\eqref{equ: minimize in poly} diverges on the boundary of the feasible domain. Hence, the optimal solution achieves in the  relative interior of $\mathcal{P}$, i.e. $x^\star _i>0$ for $i\in I$.

    We consider the first order optimal condition of the minimization problem:
    \begin{equation}
        f'(\bm{x}_0+\bm{x}^\star )=\bm{A}^\top\bm{\xi},
    \end{equation}
    with some dual optimal vector $\bm{\xi}\in \mathbb{R}^k$.
    Let $\bm{x}_c$ be the centroid of $\mathcal{P}$. Multiplying the optimal condition with $\bm{x}_c$ and $\bm{x}^\star $, we obtain:
   \begin{equation}
\bm{x}_c^{\top}f'(\bm{x}_0+\bm{x}^\star )=\bm{b}^\top\bm{\xi}={\bm{x}^\star}^{\top}f'(\bm{x}_0+\bm{x}^\star)
    \end{equation}
    Since $\bm{x}_0$ has support only on $I^C$, we have
    \begin{equation}
    \label{equ: lemma for additional 1}
     \bm{x}_c^{\top}f'(\bm{x}_0+\bm{x}^\star)=\sum_{j\in I^C}x_{c,j}f'(x_{0,j}+x^\star_j)+\sum_{j\in I}x_{c,j}f'(x^\star_j),
    \end{equation}
    and
    \begin{equation}
    (\bm{x}^\star)^{\top}f'(\bm{x}_0+\bm{x}^\star)=\sum_{j\in I^C}x^\star_{j}f'(x_{0,j}+x^\star_j)+\sum_{j\in I}x^\star_{j}f'(x^\star_j).
    \end{equation}
    Since $x_{min}>2v_{max}$, we have
   \begin{equation}
        \max_{j\in I^C}\left|f'(x_{0,j}+x^\star_j)\right|\leq |f'(x_{max}+v_{max})|\vee|f'(x_{min}-v_{max})|,
    \end{equation}
    hence
    \begin{equation}
        \left|\sum_{j\in I^C}x^\star_{j}f'(x_{0,j}+x^\star_j)\right|\vee\left|\sum_{j\in I^C}x_{c,j}f'(x_{0,j}+x^\star_j)\right|\leq mv_{max}(|f'(x_{max}+v_{max})|\vee|f'(x_{min}-v_{max})|).
    \end{equation}
    Since the function $xf'(x)$  is non-positive and is decreasing when $x\leq \delta$,
    \begin{equation}
    \label{equ: lemma for additional 2}
        0\geq\sum_{j\in I}x^\star_{j}f'(x^\star_j)\geq m v_{max}f'(v_{max}).
    \end{equation}
    Combining equation~\eqref{equ: lemma for additional 1} to~\eqref{equ: lemma for additional 2}, we have
    \begin{equation}
        0\leq-\sum_{j\in I}x_{c,j}f'(x^\star_j)\leq m (-v_{max} f' (v_{max})+2v_{max}(|f'(x_{max}+v_{max})|\vee|f'(x_{min}-v_{max})|)).
    \end{equation}
    Therefore,
    \begin{equation}
        \max_{j\in I}\left|f'(x^\star _j)\right|\leq m \frac{v_{max}}{v_{min}}( |f' (v_{max})|+2(|f'(x_{max}+v_{max})|\vee|f'(x_{min}-v_{max})|))
    \end{equation}
    
\end{proof}
\subsection{Penalty functions}
We have introduced our assumptions on the penalty function through a conjugate expression. To understand properties of the proper penalty functions, we list several results of the Legendre transformation ~\citep{rockafellar1997convex}.
\begin{definition}[Corollary 26.3.1 in \cite{rockafellar1997convex}]
    A closed proper convex $f$ is a \textit{convex function of Legendre type} if $\partial f$ is one-to-one.
\end{definition}
 The theorem characterizes the duality properties of convex functions of Legendre type.
\begin{theorem}[Theorem 26.5 in ~\cite{rockafellar1997convex}]
\label{theorem: legendre conjugate}
Let $f$ be a closed convex function. Let $C = \operatorname{int}(\operatorname{dom} f)$ 
and $C^*  = \operatorname{int}(\operatorname{dom} f^* )$. Then $(C,f)$ is a convex function of Legendre type 
if and only if $(C^* ,f^* )$ is a convex function of Legendre type. When these 
conditions hold, $(C^* ,f^* )$ is the Legendre conjugate of $(C,f)$, and $(C,f)$ 
is in turn the Legendre conjugate of $(C^* ,f^* )$. The gradient mapping $\nabla f$ is 
then one-to-one from the open convex set $C$ onto the open convex set $C^* $, 
continuous in both directions, and $\nabla f^*  = (\nabla f)^{-1}$.
\end{theorem}
Applying those results to the function $q(x)$ and $p(x)$ in assumption~\ref{new assumptions on p}, we prove the lemma~\ref{lem:p_properties}.
\begin{proof}[Proof of Lemma~\ref{lem:p_properties}]
    Since $q(x)\in C^2$ and $q''(x)>0$, $q'(x)$ is increasing in its domain and $\operatorname{cl} q$ is a convex function of Legendre type. Theorem~\ref{theorem: legendre conjugate} indicates that $p(x)$ is a convex function of Legendre type and that
    $p'=(q')^{-1}$. Thus we have
    \begin{equation}
    \label{equ in lemma 3.1}
        (-\infty,0)\subseteq \operatorname{im}q'=\operatorname{int}(\operatorname{dom}p),\quad  \operatorname{im}p'=\operatorname{int}(\operatorname{dom}(\operatorname{cl} q))= (0,+\infty).
    \end{equation}
    Also, by taking second derivative to $p(x)$ and $q(x)$, we have when $y=p'(x)$ and $y\in\operatorname{int}(\operatorname{dom}(\operatorname{cl} q))$, 
    \begin{equation}
        p''(x)=\frac{1}{q''(y)}.
    \end{equation}
    This indicates that $p''(x)$ is strict positive in its domain and locally Lipschitz by the Lipschitzness of $q''(x)$. Hence, $p'(x)$ is increasing in its domain and combined with~\eqref{equ in lemma 3.1}, we have $\lim_{x\rightarrow -\infty} p'(x)=0$.
\end{proof}

\section{Proof of main theorems}
\subsection{Bias in entropic optimal transport: proof of~\ref{proposition: ROT}}
First of all, let us restate the regularized optimal transport settings in ~\cite{KlaTamMun20}. The optimal transportation is written as:
\begin{equation}
\label{Klatt OT formulation}
\bm{\pi}^\star =\argmin_{\bm{\pi}\in\mathbb{R}^{N\times N}} \langle\bm{c},\bm{\pi}\rangle , \text{s.t.} \sum_{i=1}^N  \pi_{i,j}=s_j, \ \sum_{j=1}^N \pi_{i,j}=t_i,\ \pi_{i,j}\geq 0.
\end{equation}
The entropic regularized optimal transport with the entropic function as $f(x)=x\log x-x$ and a regularized $\lambda$ is formularized as:
\begin{equation}
\label{Klatt ROT formulation}
\bm{\pi}_\lambda=\argmin_{\bm{\pi}\in\mathbb{R}^{N\times N}} \langle\bm{c},\bm{\pi}\rangle +\lambda \sum_{i,j=1}^Nf(\pi_{i,j}) ,  \ \text{s.t.} \sum_{i=1}^N  \pi_{i,j}=s_j, \ \sum_{j=1}^N \pi_{i,j}=t_i.
\end{equation}

We restrict our analysis for the empirical problems to the unique solution case, i.e., $\bm{\pi}^\star $ is a singleton, and the one-sided case where  only $\bm{t}$ is replaced by its empirical density $\bm{t}_n$. Let $\bm{\pi}_{\lambda,n} $ denote the empirical regularized optimal transport plan and $\bm{\pi}(\bm{t}_n,\bm{s})$ the empirical true optimal plan. Let $\lambda=\lambda_n$ be a decreasing sequence of $n$.

We first examine the small regularization regime, i.e.,  $\lambda \log (\sqrt{n})\rightarrow 0$.  We prove~\eqref{equ: ROT expansion} by showing that the intrinsic distance $d(\bm{\pi}_{\lambda,n},\bm{\pi}(\bm{t}_n,\bm{s}))=O_p(\epsilon(\lambda))\sim o_p(1\sqrt{n})$.  Note that when $\bm{\pi}^\star $ is a degenerate vertex, $\bm{\bm{\pi}(\bm{t_n},s) }$ can be a convex polytope and $d(\bm{\pi}_{\lambda,n},\bm{\pi}(\bm{t}_n,\bm{s}))$ denotes the minimal distance between $\bm{\pi}_{\lambda, n}$ and the vectors in the set $\bm{\pi}(\bm{t_n},\bm{s})$. Consequently,  the regularized optimal transportation plan inherits the non-Gaussianity from $\bm{\pi}(\bm{t_n},s) $.

In the large regularization regime, where $\lambda \log (\sqrt{n})\rightarrow +\infty$, we prove our results under the special case where $\bm{\pi}^\star $ is non-generate. Under this condition, we demonstrate that  $(\bm{\pi}_{\lambda,n}-\bm{\pi}^\star )=O_p ( \bm{E}_\lambda)$ implying that $\sqrt{n}(\bm{\pi}_{\lambda,n}-\bm{\pi}^\star )$ diverges with probability $1$. Such a divergence behavior should be generalizable to the general generate case.

\subsubsection{The small regularization regime: $\lambda \log (\sqrt{n})\rightarrow 0$}
We introduce serval lemmas that describe the properties of the solutions for~\eqref{Klatt OT formulation} and~\eqref{Klatt ROT formulation} and are key to the proof for~\ref{proposition: ROT}. 
\begin{lemma}
\label{key lemma 1 for 2.1}
    Let  $\bm{\pi}^\star _0$ be one of the optimal plan for~\eqref{Klatt OT formulation} and $\bm{\pi}_\lambda$ be the unique optimal plans defined in~\eqref{Klatt ROT formulation}.  Let $\Delta \bm{\pi}=\bm{\pi}_\lambda-\bm{\pi}^\star _0$, we have 
    \begin{equation}
        \sum_{(i,j)\in\text{supp}(\bm{\pi }_0^\star )^c }f(\Delta\pi_{i,j})\leq - \frac{\eta_0}{\lambda}\sum_{(i,j)\in\text{supp}(\bm{\eta }^\star ) } \Delta\pi_{i,j}-\sum_{(i,j) \in \text{supp}(\bm{\pi}^\star _0)}f'((\bm{\pi}^\star _0)_{i,j})\Delta\pi_{i,j},
    \end{equation}
    where  $\eta_0>0$ is a constant  depending only on $\bm{c}$ and $\bm{\eta}^\star $ is the centroid of the dual optimal face (defined in~\eqref{optimality condition}) 
 for~\eqref{Klatt OT formulation}.
\end{lemma}
\begin{proof}
    Since $\bm{\pi}_0^\star $ is also feasible to~\eqref{Klatt ROT formulation}, we have
    \begin{equation}
\label{equ: ROT compare}
    \langle\bm{c}, \Delta{\bm{\pi}}\rangle +\lambda \sum_{i,j=1}^N(f((\bm{\pi}_{\lambda})_{i,j})-f((\bm{\pi}^\star _0)_{i,j}))\leq 0, \quad \sum_{i=1}^N \Delta \pi_{i,j}=0, \ \sum_{j=1}^N \Delta \pi_{i,j}=0.
\end{equation}
When $(i,j)\in \text{supp}(\bm{\pi }_0^\star )^c$, we have
\begin{equation}
f((\bm{\pi}_{\lambda})_{i,j})-f((\bm{\pi}^\star _0)_{i,j})=f(\Delta\pi_{i,j})
\end{equation}
When $(i,j)\in \text{supp}(\bm{\pi }_0^\star )$, by the convexity of $f(x)$ on $[0,+\infty)$ we have
\begin{equation}
    f((\bm{\pi}_{\lambda})_{i,j})-f((\bm{\pi}^\star _0)_{i,j})\geq f'((\bm{\pi}^\star _0)_{i,j})\Delta\pi_{i,j}
\end{equation}
By the definition of dual variable $\bm{\eta}$ in~\eqref{optimality condition} for the problem~\eqref{Klatt OT formulation}, for every feasible $\bm{\eta}$ we have 
\begin{equation}
\langle\bm{c},\Delta\bm{\pi}\rangle=\langle\bm{\eta},\Delta\bm{\pi}\rangle=\sum_{(i,j)\in\text{supp}(\bm{\eta }) }\eta_{i,j}\Delta\pi_{i,j}.
\end{equation}
Let $\{\bm{v}_i,\ i\in[N]\}$ be the vertices of the feasible polytope for $\bm{\eta}$. Since $\bm{\eta}^\star $ is the centroid of a face or is a vertex, we have
\begin{equation}
    \min_{i\in \text{supp}(\bm{\eta }^\star )}\eta^\star _i\geq \frac{1}{N} \min_i\min_{j\in \text{supp}(\bm{v}_i)}v_{i,j}\coloneqq\eta_0.
\end{equation}
Note that $\eta_0$ only depends on the geometry of the dual feasible polytope which is independent of $\bm{t}$ and $\bm{s}$.
Combining the facts above, we have the inequality in the lemma proved.
\end{proof}
Another technique employed in our analysis involves establishing bounds for Legendre-type convex functions.
\begin{lemma}
\label{key lemma 2 for 2.1}

    Let $f(\bm{x}):\mathcal{C}\rightarrow\mathbb{R}$ be a convex function of Legendre type and $g=f^* $. 
    
    Let $\bm{x}_0\in \mathcal{C}$. If $f(\bm{x}_0)\leq \langle\bm{\alpha},\bm{x}_0 \rangle$, for every $\bm{y}_0\in \mathcal{C}^* $, we have  
    \begin{equation}
        \langle\bm{y_0}-\bm{\alpha},\bm{x}_0\rangle\leq g(\bm{y_0}).
    \end{equation}
    
\end{lemma}

\begin{proof}
By Theorem~\ref{theorem: legendre conjugate}, $g(x)$ is  a convex function of Legendre type and $f=g^* $.
Therefore, for  every $\bm{y}_0\in \mathcal{C}^* $, we have 
\begin{equation}
    \langle\bm{\alpha},\bm{x}_0 \rangle\geq f(\bm{x}_0)=\max_{\bm{y}\in \mathcal{C}^* }\langle \bm{x}_0,\bm{y} \rangle-g(\bm{y})\geq \langle \bm{x}_0,\bm{y}_0 \rangle-g(\bm{y}_0).
\end{equation}
This conclude the proof.
\end{proof}

\begin{lemma}
\label{auxilary lemma for 2.1}
    Let $\bm{\pi}_n^\star \in \bm{\pi}(\bm{t}_n,\bm{s})$ be defined as 
    \begin{equation}
\label{equ: optimal center2}
\bm{\pi}^\star _n=\argmin_{\bm{\pi}\in\bm{\pi}(\bm{t}_n,\bm{s})} \sum_{i,j=1}^Nf(\pi_{i,j}).
\end{equation}
Let $\Delta\bm{\pi}\in \mathbb{R}^{N\times N}$ satisfy
\begin{equation}
\label{ot linear constraints}
    \sum_{i=1}^N \Delta \pi_{i,j}=0, \ \sum_{j=1}^N \Delta \pi_{i,j}=0.
\end{equation}
For every such vector $\Delta \bm{\pi}$, we have 
\begin{equation}
    \sum_{(i,j) \in \text{supp}(\bm{\pi}^\star _n)}f'((\bm{\pi}^\star _n)_{i,j})\Delta\pi_{i,j}=\sum_{(i,j) \in \text{supp}(\bm{\pi}^\star _n)^c}\alpha_{i,j}\Delta\pi_{i,j},
\end{equation}
with a uniform vector $\bm{\alpha}$ that only depends on $\bm{\pi}^\star _n$ and $\|\bm{\alpha}\|\leq O_p(\log n)$.
\end{lemma}
\begin{proof}
We perform separate proofs for two distinct cases:when $\bm{\pi}(\bm{t}_n,\bm{s})$ is a singleton and when it is a polytope.

\textbf{1. $\bm{\pi}(\bm{t}_n,\bm{s})$ is a singleton: }

In this case, $\bm{\pi}^\star _n $ is the unique empirical optimal transport plan. Since $\bm{\pi}_n^\star $ and $\Delta \bm{\pi}$ are constrained with the same linear operator,  we can apply Lemma~\ref{lemma: I_0 is enough} to $\Delta \bm{\pi}$ and conclude that there exists a tensor $L_{i,j,k,l}'$ with $(i,j)\in \text{supp}(\bm{\pi}^\star _n)$ and $(k,l)\in \text{supp}(\bm{\pi}^\star _n)^c$ such that
\begin{equation}
\sum_{(i,j) \in \text{supp}(\bm{\pi}^\star _n)}f'((\bm{\pi}^\star _n)_{i,j})\Delta\pi_{i,j}=    \sum_{(k,l)\in \text{supp}(\bm{\pi}^\star _n)^c} \left(\sum_{(i,j) \in \text{supp}(\bm{\pi}^\star _n)}f'((\bm{\pi}^\star _n)_{i,j})L'_{i,j,k,l}\right) \Delta \bm{\pi}_{k,l}.
\end{equation}
Therefore, $\alpha_{k,l}=\sum_{(i,j) \in \text{supp}(\bm{\pi}^\star _n)}f'((\bm{\pi}^\star _n)_{i,j})L'_{i,j,k,l}$ and there exists a constant $C$, $$\|\alpha\|\leq C\max_{(i,j)\in \text{supp}(\bm{\pi}^\star _n)}\|f'((\bm{\pi}^\star _n)_{i,j})\|.$$
Since $(\bm{\pi}^\star _n)_{i,j}\geq C'\|\bm{t}_n-\bm{t}\| $ with a constant $C'$ ,  we have $\|f'((\bm{\pi}^\star _n)_{i,j})\|=O_p(\log n)$.

\textbf{2. $\bm{\pi}(\bm{t}_n,\bm{s})$ is a non-degenerate polytope:}

Let $S=\cup_{\bm{\pi}\in \bm{\pi}(\bm{t}_n,s)}\text{supp}(\bm{\pi})$. 
Since $f'(x)$ diverges on the boundary of $\bm{\pi}(\bm{t}_n,s)$, we have $\text{supp}(\bm{\pi}^\star _n)=S$. 
The first-order optimality condition for~\eqref{equ: optimal center2}  writes 
\begin{equation}
    f'((\bm{\pi}_n^\star )_{i,j})=\phi_i+\psi_j \text{ for }(i,j)\in S, 
\end{equation}
with some vector $\bm{\phi},\bm{\psi}\in \mathbb{R}^N$ linearly with respect to $f'((\bm{\pi}_n^\star ))$.

Therefore,
\begin{equation}
    \sum_{(i,j) \in \text{supp}(\bm{\pi}^\star _n)}f'((\bm{\pi}^\star _n)_{i,j})\Delta\pi_{i,j}= \sum_{(i,j) \in \text{supp}(\bm{\pi}^\star _n)}(\phi_i+\psi_j)\Delta\pi_{i,j}=-\sum_{(i,j) \in \text{supp}(\bm{\pi}^\star _n)^c}(\phi_i+\psi_j)\Delta\pi_{i,j}.
\end{equation}
This is because of $\sum_{i,j=1}^N(\phi_i+\psi_j)\Delta\pi_{i,j}=0$ by the linear constraints for $\Delta\bm{\pi}$.  
Give in equation~\eqref{appendix: xb' expression}, the vertices of $\bm{\pi}(\bm{t}_n,s)$ are feasible basis solutions. Applying~\ref{lemma for 2.1} with $\mathcal{P}=\bm{\pi}(\bm{t}_n,s)-\bm{\pi}^\star $ , $m=|S|$, and $\bm{x}_0=\bm{\pi}^\star $, we conclude that$\|f'((\bm{\pi}^\star _n)_{i,j})\|=O_p(\log n)$ so that $|\alpha_{i,j}|=|\phi_i+\psi_j|=O_p(\log n)$.

\end{proof}
Utilizing the lemmas above, we can prove the small $\lambda$ regime in proposition~\ref{proposition: ROT}
\begin{proposition}
\label{prop1 for 2.1}
    In the regime of $\lambda\log(\sqrt{n})\rightarrow 0$, we have $d(\bm{\pi}_{\lambda,n},\bm{\pi}(\bm{t}_n,\bm{s}))=O_p(\epsilon(\lambda))$.
\end{proposition}
\begin{proof}
    Let $\bm{\pi}_n^\star $ be defined in Lemma~\ref{auxilary lemma for 2.1} as the solution to  a convex minimization problem over $\bm{\pi}(\bm{t}_n,\bm{s}))$. Let $\bm{\eta}^\star _n$ be a dual variable defined in equation~\eqref{optimality condition} that is also the centroid of the dual optimal face for $\bm{\pi}(\bm{t}_n,\bm{s})$.

     When $\bm{\pi}(\bm{t}_n,\bm{s})$ forms a non-degenerate polytope, Lemma~\ref{auxilary lemma for 2.1} establishes that $\bm{\pi}_n^\star $ is achieved within the relative interior of this polytope. Consequently, strict complementary slackness is satisfied between $\bm{\pi}_n^\star $ and $\bm{\eta}^\star _n$ , i.e.,
     \begin{equation}
         \label{equ: dual}
    \text{supp}(\bm{\eta }_n^\star )=\text{supp}(\bm{\pi}^{\star }_{n})^c.
     \end{equation}
    Let $\Delta \bm{\pi}=\bm{\pi}_{\lambda,n}-\bm{\pi}^\star _n$.
     Applying Lemma~\ref{key lemma 1 for 2.1} for $\bm{\pi}^\star _n$ and $\bm{\pi}_{\lambda,n}$ then we have
     \begin{equation}
        \sum_{(i,j)\in\text{supp}(\bm{\eta }_n^\star ) }f(\Delta\pi_{i,j})\leq - \frac{\eta_0}{\lambda}\sum_{(i,j)\in\text{supp}(\bm{\eta }^\star _n) } \Delta\pi_{i,j}-\sum_{(i,j) \in \text{supp}(\bm{\pi}^\star _n)}f'((\bm{\pi}^\star _n)_{i,j})\Delta\pi_{i,j}.
    \end{equation}
    Applying Lemma~\ref{auxilary lemma for 2.1} to $f'((\bm{\pi}^\star _n)_{i,j})$ in the inequality above we have:
    \begin{equation}
        \sum_{(i,j)\in\text{supp}(\bm{\eta }_n^\star ) }f(\Delta\pi_{i,j})\leq - \sum_{(i,j)\in\text{supp}(\bm{\eta }^\star _n) } (\frac{\eta_0}{\lambda}+\alpha_{i,j})\Delta\pi_{i,j},
    \end{equation}
    where $\|\alpha\|\leq O_p(\log n)=o_p(1/\lambda)$.

    Applying Lemma~\ref{key lemma 2 for 2.1} to the entropic regularizer $f$ with dual function $f^\star (y)=\exp(y)$, by taking $y_{0,i,j}= 1-\alpha_{i,j}- \frac{\eta_0}{\lambda}$, we have 
    \begin{equation}
        \sum_{(i,j)\in\text{supp}(\bm{\eta}_n) }\Delta\pi_{i,j}\leq \sum_{(i,j)\in\text{supp}(\bm{\eta}_n) }\exp(1-\alpha_{i,j}- \frac{\eta_0}{\lambda})\leq \exp(-\frac{C_0\eta_0}{\lambda}).
\end{equation}
Here $C_0$ is a constant independent of $n$, $\lambda$, and $\eta_0$.
Let $\mathcal{D}_n$ be the feasible set of program~\eqref{Klatt OT formulation} with $\bm{t}_n$.
We can write the optimal face $\bm{\pi}(\bm{t_n}, \bm{s})$ as
\begin{equation}
    \bm{\pi}(\bm{t_n}, \bm{s})=\{\bm{\pi}\in \mathcal{D}| \pi_{i,j} =0, (i,j)\in \text{supp} (\bm{\eta}_n^\star )\}.
\end{equation}
And  $\bm{\pi}_{\lambda,n}$ belongs to the convex set:
 \begin{equation}
   \bm{\pi}'(\bm{t_n,s})\coloneqq \{\bm{\pi}\in \mathcal{D}| \pi_{i,j} =\Delta\pi_{i,j}, (i,j)\in \text{supp} (\bm{\eta}_n^\star )\}.
\end{equation}
Therefore, 
due to the Lipschitzian property of polytopes proved in~\cite{walkup1969lipschitzian}, 
\begin{equation}
    d(\bm{\pi}_{\lambda,n},\bm{\pi}(\bm{t}_n,\bm{s}))\leq d_H(\bm{\pi}'(\bm{t_n,s}),\bm{\pi}(\bm{t}_n,\bm{s})))\leq C'\exp(-\frac{C_0\eta_0}{\lambda}),
\end{equation}
where $d_H(\cdot,\cdot)$ denotes the Hausdorff distance between two convex polytopes.
\end{proof}
\subsubsection{The large regularization regime: $\lambda \log (\sqrt{n})\rightarrow +\infty$}
For simplicity of the proof, we inexplicitly write the equality constraints in~\eqref{Klatt OT formulation} as
\begin{equation}
    \bm{A}_0\bm{\pi}=(\bm{t},\bm{s}),
\end{equation}
with a matrix $\bm{A}_0\in \mathbb{R}^{N^2\times 2N}$.
We consider the special case that $\bm{\pi}^\star $ is a non-degenerate basic solution (definition in Section~\ref{sec:preliminary}) with a unique basis $I^\star $, hence  
$\text{supp} (\bm{\pi}^\star )=I^\star $.
We make our prove in two steps: 1) Give expression of $\bm{E}(\lambda)$ and showing that $\lambda\log\|\bm{E}(\lambda)\|=\Theta(1)$; 2) Give the expression of the Matrix $\bm{M}$ and show that residual term $\epsilon(\lambda)$ in~\eqref{equ: ROT expansion large regularizer} is exponentially small.
\begin{proposition}
\label{Prop: first prop for large regime}
    Let $\bm{E}(\lambda)=\bm{\pi}_\lambda-\bm{\pi}^\star $. We have $\lambda\log\|\bm{E}(\lambda)\|=\Theta(1)$.
\end{proposition}
\begin{proof}
    We consider the dual program for~\eqref{Klatt ROT formulation}:
    \begin{equation}
        \bm{\xi}_\lambda=\argmin_{\bm{\xi}\in\mathbb{R}^{2N}}\langle\bm{\xi},(\bm{t},\bm{s})\rangle+\lambda\exp(-\frac{\bm{c}-\bm{A}_0^\top\bm{\xi}}{\lambda}).
    \end{equation}
    This dual minimization problem can be transform to the formulation in Theorem~\ref{theorem: unique solution general penalty} through a linear transform $\bm{\eta}=\bm{c}-\bm{A}_0^\top\bm{\xi}$.  
    
    Our theorem indicates that the minimizer satisfies
    \begin{equation}
        \bm{c}-\bm{A}_0^\top\bm{\xi}_\lambda\coloneqq\bm{\eta}_\lambda=\bm{\eta}^\star +\lambda\bm{d}^\star +\beta(\lambda),
    \end{equation}
    where $\beta(\lambda)\ll \lambda$ and  $\bm{\eta}^\star $ is the dual optimal solution for~\eqref{Klatt OT formulation}.  Since $\bm{\pi}^\star $ is the unique non-degenerate solution,  the primal--dual solution pair $(\bm{\pi}^\star ,\bm{\eta}^\star )$ satisfies the strict complementary slackness.

    We can express the solution $\bm{\pi}_\lambda$ using the dual program solution:
    \begin{equation}
        \bm{\pi}_\lambda=\exp(-\frac{\bm{\eta}_\lambda}{\lambda}).
    \end{equation}
    Since $\bm{A}_0\bm{E}(\lambda)=0$ and the unique basis $I^\star =\text{supp}(\bm{\pi}^\star )$, by Lemma~\ref{lemma: I_0 is enough}, we know that $\bm{E}(\lambda)_{I^\star }$ is a linear function of $\bm{E}(\lambda)_{I^{\star c}}$. Thus, we only need to show that $\bm{E}(\lambda)_{I^{\star c}}$ is exponentially small.

    Since $\text{supp}(\bm{\pi}^\star )^c=\text{supp}(\bm{\eta}^\star )$, $\bm{E}(\bm{\lambda})_{I^{\star c}}=(\bm{\pi}_\lambda)_{I^{\star c}}$ and
    \begin{equation}
        \lambda\log(\bm{E}(\bm{\lambda})_{I^{\star c}})=-(\bm{\eta}^\star +\lambda\bm{d}^\star +\beta(\lambda))_{I^{\star c}}=\Theta(1).
    \end{equation}
    \end{proof}
    \begin{proposition}
        Let $\bm{M}_{\Delta\bm{t}}\coloneqq\bm{M}(\bm{t_n-\bm{t}})$ be defined as 
        \begin{equation}
            (\bm{M}_{\Delta\bm{t}})_{I^\star }=\bm{A}_{0,I^\star }^{-1}(\bm{t}_n-\bm{t},0),\quad (\bm{M}_{\Delta\bm{t}})_{I^{\star c}}=0,
        \end{equation}
        and $\bm{\pi}'_{n,\lambda}=\bm{\pi}_\lambda+\bm{M}_{\Delta\bm{t}}$.

        We have $\Delta \bm{\pi}\coloneqq\bm{\pi}_{\lambda,n}-\bm{\pi}'_{n,\lambda}=O_p(\|\bm{t_n}-\bm{t}\|\epsilon_{\lambda})$, where $\epsilon_{\lambda}$ converge exponentially fast to $0$.
    \end{proposition}
    \begin{proof}
    By the definition of $\Delta\bm{\pi}$, we have $\bm{A}_0\Delta\bm{\pi}=0$.

    In the empirical situation with $\bm{t}$ replaced by $\bm{t}_n$, we evaluate the difference in the cost function in equation ~\eqref{Klatt ROT formulation}  at two different vectors: the optimal plan $\bm{\pi}_{\lambda,n}$ and a feasible plan $\bm{\pi}'_{\lambda,n}$. Since $\bm{\pi}_{\lambda,n}$ is optimal for the program, we have
    \begin{equation}
    \label{governing inequality to prove large regime}
        \langle\bm{c}, \Delta{\bm{\pi}}\rangle +\lambda \sum_{i,j=1}^N f\left((\bm{\pi}_{\lambda,n})_{i,j}\right)-f\left((\bm{\pi}'_{\lambda,n})_{i,j}\right)\leq 0
    \end{equation}
    
        The first order optimality condition for~\eqref{Klatt ROT formulation} writes:
        \begin{equation}
            \bm{c}+\lambda f'(\bm{\pi}_\lambda)=\bm{A_0^\top\xi}.
        \end{equation}
        Therefore, we have 
        \begin{equation}
            \langle\bm{c},\Delta\bm{\pi}\rangle=\langle-\lambda f'(\bm{\pi}_\lambda),\Delta\bm{\pi}\rangle,
        \end{equation}
        and the governing inequality~\eqref{governing inequality to prove large regime} can be written as
        \begin{equation}
             \sum_{i,j=1}^N h_{i,j}\coloneqq\sum_{i,j=1}^N f\left((\bm{\pi}_{\lambda,n})_{i,j}\right)-f\left((\bm{\pi}'_{\lambda,n})_{i,j}\right)-f'\left((\bm{\pi}_{\lambda})_{i,j}\right)\Delta\bm{\pi}_{i,j}\leq 0
        \end{equation}
        We separately simplify the summations over the set $I^\star $ and its complement.

        For $(i,j)\in I^\star $, by the convexity of $f$ we have 
        $$h_{i,j}\geq (f'\left((\bm{\pi}'_{\lambda,n})_{i,j}\right)-f'\left((\bm{\pi}_{\lambda})_{i,j}\right))\Delta\bm{\pi}_{i,j}=\Delta\bm{\pi}_{i,j}\log\left(\frac{(\bm{\pi}'_{\lambda,n})_{i,j}}{(\bm{\pi}_{\lambda})_{i,j}}\right)$$

        For $(i,j)\in I^{\star c}$, since $\text{supp} (\bm{M}_{\Delta\bm{t}})\subseteq I^\star  $, we have $(\bm{\pi}'_{\lambda,n})_{i,j}=(\bm{\pi}_{\lambda})_{i,j}$ and 
        \begin{align*}
            h_{i,j}&=f\left((\bm{\pi}_{\lambda}+\Delta\bm{\pi})_{i,j}\right)-f\left((\bm{\pi}_{\lambda})_{i,j}\right)-f'\left((\bm{\pi}_{\lambda})_{i,j}\right)\Delta\bm{\pi}_{i,j} \\
            &=(\bm{\pi}_\lambda)_{i,j}\left[f\left(1+\frac{\Delta\bm{\pi}_{i,j}}{(\bm{\pi}_\lambda)_{i,j}}\right)-f(1)\right]\\
            &\geq C(\bm{\pi}_\lambda)_{i,j}\min\left(\left|\frac{\Delta\bm{\pi}_{i,j}}{(\bm{\pi}_\lambda)_{i,j}}\right|,\left(\frac{\Delta\bm{\pi}_{i,j}}{(\bm{\pi}_\lambda)_{i,j}}\right)^2\right),\end{align*}
            where $C$ is a fixed constant.
        The second equality is obtained utilizing properties of the entropy function that $f(x)=x\log x-x$, $f'(x)=\log x$, and $f(1)=-1$. The last inequality is a direct result of Lemma~\ref{lem:new_norm_bound}.

        Finally, we have
        \begin{equation}
            C\sum_{(i,j)\in I^{\star c}}\min\left(\left|{\Delta\bm{\pi}_{i,j}}\right|,\frac{(\Delta\bm{\pi}_{i,j})^2}{(\bm{\pi}_\lambda)_{i,j}}\right)\leq-\sum_{(i,j)\in I^{\star }}\Delta\bm{\pi}_{i,j}\log\left(\frac{(\bm{\pi}'_{\lambda,n})_{i,j}}{(\bm{\pi}_{\lambda})_{i,j}}\right).
        \end{equation}
        
        We then generate a bound for the right hand side of the inequality.

        When $(i,j)\in I^{\star }$, $(\bm{\pi}_\lambda)_{i,j}=\Theta(1)$, thus $$\log\left(\frac{(\bm{\pi}'_{\lambda,n})_{i,j}}{(\bm{\pi}_{\lambda})_{i,j}}\right)=\log\left(1+\frac{(\bm{M}_{\Delta \bm{t}})_{i,j}}{(\bm{\pi}_{\lambda})_{i,j}}\right)=O_p(\|\bm{M}_{\Delta \bm{t}}\|)\ll1.$$
        By Lemma~\ref{lemma: I_0 is enough}, we have $\|(\Delta\bm{\pi})_{I^\star }\|\leq C'\|(\Delta\bm{\pi})_{I^{\star c}}\|$.

        Consequently,  
        \begin{equation}
            \sum_{(i,j)\in I^{\star c}}\min\left(\left|{\Delta\bm{\pi}_{i,j}}\right|,\frac{(\Delta\bm{\pi}_{i,j})^2}{(\bm{\pi}_\lambda)_{i,j}}\right)\leq C''\|\bm{M}_{\Delta \bm{t}}\|\|(\Delta\bm{\pi})_{I^{\star c}}\|,
        \end{equation}
         which indicates that 
        \begin{equation}
            \min\left(\Delta\bm{\pi}_{max},\frac{(\Delta\bm{\pi}_{max})^2}{(\bm{\pi}_\lambda)_{i,j}}\right)\leq C''N^2\|\bm{M}_{\Delta \bm{t}}\|\Delta\bm{\pi}_{max}
        \end{equation}
        with $\Delta\bm{\pi}_{max}=\max_{(i,j)\in I^{\star c}}\left|{\Delta\bm{\pi}_{i,j}}\right|$. 
       Since $\|\bm{M}_{\Delta \bm{t}}\|=o_p(1)$, we conclude that 
        \begin{equation}
            \|(\Delta\bm{\pi})_{I^{\star c}}\|\lesssim\|(\bm{\pi}_\lambda)_{I^{\star c}}\|\|\bm{M}_{\Delta \bm{t}}\|=O_p((\bm{t}_n-\bm{t})E(\lambda)).
        \end{equation}
    \end{proof}

\subsection{Convergence of penalized linear program: proof of Theorem~\ref{theorem: unique solution general penalty}}
\label{sec:theorem_proof}

Theorem~\ref{theorem: unique solution general penalty} establishes the convergence behavior of solutions to the penalized program as the penalty parameter approaches zero. The asymptotic expansion in the theorem is primarily characterized by the solution for an auxiliary convex  program~\eqref{equ for 1st order}. Before we prove the this theorem, we first prove the existence and uniqueness of the penalized program and the auxiliary program solutions.  
\begin{proof}[Proof of the existence and uniqueness of the penalized program solution (Proposition~\ref{existence and uniqueness})]

    $ $
    
    By Assumption~\ref{assumption on LP}, there is a feasible vector $\bm{x}_0 $ for~\eqref{primal} with all positive entries.
    In particular, $\bm{x}_0 \in \inter \dom f_r$, so $\bm x_0$ is feasible for~\eqref{equ:penalized function} as well.
	Lemma~\ref{lem:p_properties} implies that $f_r(\bx)$ is strictly convex, so, if a solution exists, it is unique.
    
	It therefore suffices to show that a solution exists.
	This follows from convex duality: the penalized program~\eqref{general penalty problem} is the convex dual of
    \begin{equation}
    \label{dual penalized program}
        \max_{\bm{\lambda}\in \mathbb{R}^k} \langle \bb, \bm{\lambda} \rangle - r \sum_i q(c_i - (\bA^\top \bm{\lambda})_i).
    \end{equation}
    Fenchel duality~\citep[Proposition 5.3.8]{bertsekas2009convex} implies that a primal solution is guaranteed to exist so long as Slater's condition for the dual problem holds, that is, so long as there exists $\bm{\lambda}^\star $ such that $\bc - A^\top \bm{\lambda}^\star  > 0$.
    But this follows from the fact that the set of solutions to the original linear program~\eqref{primal} is bounded~\citep[Corollary 7.1k]{schrijver1998theory}.

	Finally, we show that the solution satisfies $\bx(r, \bb) \in \inter \dom f_r$.
Let $t = \sup\{s: s \in \dom p\}$.
Since $\dom f_r = (\dom p)^m$, if the $\bx(r, \bb)$ does not lie in $\inter \dom f_r$, then there exists a coordinate $i$ such that $\bx(r, \bb)_i = t$.
Lemma~\ref{lem:p_properties} shows that $t \geq 0$ and $p'(x) \to +\infty$ as $x \uparrow t$; in particular, this implies that the directional derivative of $f_r$ at $\bx(r, \bb)$ in the direction $\bx_0 - \bx(r, \bb)$ is $-\infty$, contradicting the claimed optimality of $\bx(r, \bb)$.
%
%
%
%
%
%
   
\end{proof}

To prove that the auxiliary program has a unique solution, we show that it is locally strong convex. Here, $I_0$ is defined the same as in~\eqref{equ for 1st order}.
\begin{lemma}[Locally strong convexity of the auxiliary program]
 \label{lemma: g convex on affine}
 Let $g(\bm{x})=\langle \bc,\bx\rangle+ \sum_{i\in I_0} p(-x_i)$ and the affine space $\mathbb{A}=\{\bm{x}\in\mathbb{R}^m|\bm{Ax}=\bm{\alpha}_0\}$.
 Then $g(x)$ is locally strongly convex and $C^2$ on $\mathbb{A}$, i.e., for every $\bm{x}\in \operatorname{int}(\operatorname{dom}g) $, $$\inf_{\bm w \in \text{Nul } \bm{A}, \|\bm w\|^2 = 1}\langle \bm w, (\nabla^2 g) \bm w\rangle > 0.$$
 \end{lemma}
 \begin{proof}
Directly taking derivatives to $g$, we have 
    $$\langle \bm w, (\nabla^2 g) \bm w\rangle=\sum_{i\in I_0}\omega_i^2p''(x_i)\geq\min_{i\in I_0} p''(-x_i) \sum_{i\in I_0} \omega_i^2.$$
     Consider the quadratic program 
     \begin{equation}
         v_{min}=\min_{\|\bm w\|^2=1} \sum_{i\in I_0} \omega_i^2,\quad  \bm{A\omega}=0.
     \end{equation}
     The program attains a solution in the set $\|\bm w\|^2=1$ by the compactness of the feasible set.
     Since by Lemma~\ref{lemma: I_0 is enough}, $w_i=0$ for all $i\in I_0$ indicates that $\bm{w}=0$ which is not feasible for the quadratic program. Therefore, $v_{min}>0$ and by the locally strong convexity of $p(x)$ we have $$\langle \bm w, (\nabla^2 g) \bm w\rangle\geq v_{min}\min_{i\in I_0} p''(-x_i)>0.$$
 \end{proof}
 \begin{lemma}
    The minimization problem in~\eqref{equ for 1st order} is feasible and has a unique solution
\end{lemma}
\begin{proof}
\textbf{Feasibility:} 
By the Slater's condition of the original linear program~\eqref{primal}, there is a vector $\bm{x}_s$ with all positive entries that is feasible to the linear program. It is easy to verify that  $\bm{d}_0=\bm{x}_s-\bm{x}^\star $ is feasible to~\eqref{equ for 1st order} since $(\bm{d}_0)_{I_0}\geq 0$.

\textbf{Existence and uniqueness:} 
Following a similar procedure as in the proof of Proposition~\ref{existence and uniqueness}, we write the dual program of~\eqref{equ for 1st order}:
\begin{equation}
\max_{\bm{\lambda}\in\mathbb{R}^k} -\sum_{i\in I_0} q(c_i-(\bm{A}^\top\bm{\lambda})_i),\quad c_i-(\bm{A}^\top\bm{\lambda})_i=0 \text{ for } i\in I_0^c.
\end{equation}
Feasibility of the dual program, i.e. exists $\bm{\lambda}$ such that $c_i-(\bm{A}^\top\bm{\lambda})_i>0$ for $i\in I_0$ and $c_i-(\bm{A}^\top\bm{\lambda})_i=0$ for $i\in I_0^c$, is a direct result of the complementary slackness of the linear program solutions \citep[Section 7.9 ]{schrijver1998theory}. Therefore, under the same arguments as in the proof of Proposition~\ref{existence and uniqueness} we have the existence of a solution for~\eqref{equ for 1st order}. Also, since the cost function of the primal program is strictly convex on the affine space $\{\bm{d}\in \mathbb{R}^m| \bm{Ad}=0\}$ by Lemma~\ref{lemma: g convex on affine}, we can conclude that such a solution is unique.
\end{proof}

We now proceed with proving the theorem. Our approach leverages two key facts: both $\bm{x}^\star +r\bm{d}^\star $ and $\bm{x}(r,\bm{b})$  are feasible solutions to the penalized program, and $\bm{x}(r,\bm{b})$ achieves optimality. The proof employs an important property of locally strongly convex functions established in Lemma~\ref{lem:new_norm_bound}.
\begin{proof}[Proof of Theorem~\ref{theorem: unique solution general penalty}]
	Define $\bd(r) =(\bm{x}(r, \bb)-\bm{x}^\star )/r$.
	It suffices to show that $\|\bd^\star  - \bd(r)\| = O(\beta(r))$ as $r \to 0$.
	
	Define $g_r(\bd) =  \langle \bc, \bd \rangle + \sum_{i=1}^m p(-\frac{x_i^\star }{r} - d_i)$, and let $g(\bd) = \langle \bc, \bd \rangle+\sum_{i \in I_0} p( - d_i)$.
	Define $\bd(r) =(\bm{x}(r, \bb)-\bm{x}^\star )/r$.
	The convexity of $p$ implies that
	\begin{equation*}
		h_r(\bd) := g_r(\bd) - g(\bd) = \sum_{i \notin I_0} p(-\frac{x_i^\star }{r} - d_i)
	\end{equation*}
	is a convex function of $\bd$.
	Hence
	\begin{equation*}
		h_r(\bd^\star ) - h_r(\bd(r)) \leq \langle \nabla h_r(\bd^\star ), \bd^\star  - \bd(r) \rangle\,.
	\end{equation*}
	On the other hand, since $\bd(r)$ minimizes $g_r$, we also have
	\begin{equation*}
		g(\bd(r)) - g(\bd^\star ) \leq  h_r(\bd^\star ) - h_r(\bd(r))
	\end{equation*}
	Combining these facts yields
	\begin{equation*}
		g(\bd(r)) - g(\bd^\star ) \leq \langle \nabla h_r(\bd^\star ), \bd^\star  - \bd(r) \rangle\,.
	\end{equation*}
	
	For any $\bv \in \Rset^m$, we have
	\begin{equation}
		|\langle \nabla h_r(\bd^\star ), \bv \rangle| \leq \sum_{i \notin I_0} p'(-\frac{x_i^\star }{r} - d^\star _i)| v_i| = O(\beta(r) \|v_i\|) \quad \text{as $r \to 0$}\,,
	\end{equation}
	where the implicit constant depends on $\bx^\star $ but can taken independent of $r$ and $\bv$.
	
	We obtain that there an exists a positive constant $C'$ such that
	\begin{equation*}
		g(\bd(r)) - g(\bd^\star )) \leq C' \beta(r) \|\bd^\star  - \bd(r)\|\,.
	\end{equation*}
	Applying Lemma~\ref{lem:new_norm_bound} and Lemma~\ref{lemma: g convex on affine} with $K = \{\bd^\star \}$ and with $\mathbb{A}=\{\bm{x}\in\mathbb{R}^m|\bm{Ax}=\bm{0}\}$ and using the fact that $\bd^\star $ is a minimizer of $g$ over $\ker A$, we obtain
	\begin{equation}
		C \min\{\|\bd^\star  - \bd(r)\|,\|\bd^\star  - \bd(r)\|^2\}  \leq C' \beta(r) \|\bd^\star  - \bd(r)\|\,,
	\end{equation}
	which implies for $r$ small enough that $\|\bd^\star  - \bd(r)\| \leq \tfrac{C'}{C} \beta(r) = O(\beta(r))$, as desired.
    \end{proof}


%
\subsection{Asymptotic law of random penalized program solutions: proof of theorems in Section~\ref{sec: trajectory of random LP},~\ref{sec:confidence set},and~\ref{sec: bootstrap}}
This section provides the complete proofs for the main theoretical results in this paper. We first establish the asymptotic behavior of the  penalized program under small non-random perturbations of the equality constraints. We then extend these results to scenarios involving random constraints and prove the convergence law of the extrapolated estimator.
\subsubsection{Asymptotic behavior of perturbed penalized program}
We first show that the solution for the perturbed penalized program~\eqref{eq:rig_asymptotic} exists uniquely when the perturbation on vector $\bm{b}$ is small enough.
\begin{proof}[Proof of Proposition~\ref{prop:peturb_feasibility}]
    We only prove feasibility. The existence and uniqueness of the solution follows from  arguments identical to that presented in Proposition~\ref{existence and uniqueness}.

    By the Slater's condition for the linear program  in Assumption~\ref{assumption on LP}, there exists a vector $\bm{x}_0$ such that
    \begin{equation}
        \bm{Ax}_0=\bm{b},\quad x_{0,i}>0.
    \end{equation}
    Let $\bm{A}_0\in \mathbb{R}^{k\times k}$ be a full rank submatrix of $\bm{A}$ obtained by selecting columns corresponding to an index set $I$.
    Let $\Delta \bm{x}\in \mathbb{R}^m$ be constructed such that $(\Delta \bm{x})_I=\bm{A}_0^{-1}(\bm{b}'-\bm{b})$ and $(\Delta \bm{x})_{I^C}=0$.

    If  $\|\bm{b}'-\bm{b}\|\leq \|\bm{A}_0^{-1}\|\min_{i\in\{1,..,m\}}x_{0,i}$,
    then $\|\Delta \bm{x}\|\leq \min_{i\in\{1,..,m\}}x_{0,i}$, ensuring that  $\bm{x}'=\bm{x}_0+\Delta \bm{x}$ has all positive entries and satisfies $\bm{Ax}'=\bm{b}'$, thereby remaining feasible with $\bm{b}'$.  
\end{proof}

The matrix $\bm{M}^\star$ serves as the key component in the expansion trajectory of the perturbed program~\eqref{eq:rig_asymptotic}. We prove its properties stated in Proposition~\ref{proposition: unique solution of control equation}, which is crucial in the proof of Theorem~\ref{thm: trajectory with noise}.

\begin{proof}[Proof of Proposition~\ref{proposition: unique solution of control equation}]


    \textbf{Existence:} 
 Feasiblity of the program and the non-negativity of $p''(-\bm{d}^\star )$ indicate existence of the solutions.

    \textbf{Uniqueness:} Assume there are two different vectors $\bm{x}_{1}$  and $\bm{x}_2$ both solving the program~\eqref{eq:mdef}. Let $\Delta\bm{x}=\bm{x}_1-\bm{x}_2$.  Then $\Delta \bm{x}$ lies in the kernel of $\bm{A}$. The strict convexity of quadratic functions $f(x_i)=p''(-d^\star _i)x_i^2$ for $i\in I_0$ indicates that  $\Delta x_i=0$ for $i\in I_0$. Moreover, due to Lemma~\ref{lemma: I_0 is enough}, entries on $I_0$ uniquely determine $\Delta \bm{x}$ and we hence conclude $\Delta \bm{x}=0$.

    The two properties are a direct result of the first order optimality condition for this quadratic program. Let $\bm{\mu}^\star \in \mathbb{R}^k$ and $\bm{x}^\star \in \mathbb{R}^m$ be the optimal dual and primal variables separately. The first order condition indicates that
    \begin{equation}
        \bm{Ax}^\star =\bm{y},\quad \text{diag}(p''(-\bm{d}^\star )_{I_0})\bm{x}^\star -\bm{A}^\top\bm{\mu}^\star =0.
    \end{equation}
    Hence, $(\bm{x}^\star ,\bm{\mu}^\star )$ is a linear function of $(\bm{y},\bm{0})$. Therefore, there is a matrix representation $\bm{M}^\star \in \mathbb{R}^{m\times k}$ such that $\bm{x}^\star =\bm{M^\star y}$.

    Since for every $\bm{y}\in\mathbb{R}^k$ and every $\bm{q}\in \text{ker}(\bm{A})$, 
    $$\bm{A}\bm{x}^\star =\bm{AM^\star y}=\bm{y},$$
    $$
        0=\bm{q}^\top(\bm{\Sigma}^\top\bm{x}^\star -\bm{A}^\top\bm{\mu}^\star )=\bm{q}^\top\bm{\Sigma}^\top\bm{ x^\star }=\bm{q}^\top\bm{\Sigma }^\top\bm{M^\star y},
    $$
    we have $\bm{AM^\star }=I_k$ and $\bm{q}\in\text{ker}( {\bm{M}^{\star }}^\top\bm{\Sigma}^\top)$.

    Therefore, we conclude that $\bm{M}^\star  $ is a right inverse of $\bm{A}$ and that $\text{ker}(\bm{A})\subseteq \text{ker}( {\bm{M}^{\star }}^\top\bm{\Sigma}^\top)$.    
\end{proof}

Building on these results, we now characterize the trajectory of the penalized program under a small perturbation on $\bm{b}$. The techniques employed are similar to that in the proof of Theorem~\ref{theorem: unique solution general penalty}. We utilize the fact that $(\bx^\star  + r \bd^\star +\bm{M}_{\bx^\star }(\bb' - \bb))$ is feasible for the perturbed penalized program and apply Lemma~\ref{lem:new_norm_bound}. 

\begin{proof}[Proof of Theorem~\ref{thm: trajectory with noise}]
	We proceed as in the proof of Theorem~\ref{theorem: unique solution general penalty}.
	Define $\be(r, \bb') = (\bx(r, \bb') - \bx^\star  - r \bd^\star )/r$.
	Similarly, let $\be^\star  = \bm{M}^\star(\bb' - \bb)/r$.
	Our aim is to show that $\|\be(r, \bb') - \be^\star \| = O(\beta(r) + \|\bb' - \bb\|^2/r^2)$.
	
	Since $\bm{M}_{\bx^\star }$ is a right-inverse of $\bA$ by Proposition~\ref{proposition: unique solution of control equation}, we have that $\bA \be(r, \bb') = \bA \be^\star  = (\bb' - \bb)/r$.
	Therefore, as in the proof of Theorem~\ref{theorem: unique solution general penalty}, we have that $g_r(\bd^\star  + \be(r, \bb')) \leq g_r(\bd^\star  + \be^\star )$.
	Following the steps of that proof, we obtain
	\begin{equation*}
		g(\bd^\star  + \be(r, \bb')) - g(\bd^\star  + \be^\star ) \leq \langle \nabla h_r(\bd^\star  + \be^\star ),  \be^\star  - \be(r, \bb') \rangle = O(\beta(r) \|\be^\star  - \be(r, \bb')\|)\,,
	\end{equation*}
	and therefore
	\begin{equation*}
		g(\bd^\star  + \be(r, \bb')) - g(\bd^\star  + \be^\star ) - \langle \nabla  g(\bd^\star  + \be^\star ), \be(r, \bb') - \be^\star  \rangle \leq \langle \nabla  g(\bd^\star  + \be^\star ), \be^\star  - \be(r, \bb') \rangle + O(\beta(r) \|\be^\star  - \be(r, \bb')\|)\,.
	\end{equation*}
	Since $p''$ is locally Lipschitz by Lemma~\ref{lem:p_properties}, Taylor' theorem implies
	\begin{equation*}
		\langle \nabla  g(\bd^\star  + \be^\star ), \be(r, \bb') - \be^\star  \rangle = \langle \nabla g(\bd^\star ), \be(r, \bb') - \be^\star  \rangle + \langle \nabla^2 g(\bd^\star ) \be^\star , \be(r, \bb') - \be^\star  \rangle + O(\|\be^\star \|^2 \|\be(r, \bb') - \be^\star \|)\,.
	\end{equation*}
	The first-order optimality conditions of $\bd^\star $ show that the first term vanishes.
	Moreover, since
	\begin{equation*}
		\nabla^2 g(\bd^\star ) \be^\star  = r^{-1} \bm{\Sigma} \bm{M}^\star (\bb' - \bb)
	\end{equation*}
	and $\be(r, \bb') - \be^\star  $ lies in the kernel of $\bA$, Proposition~\ref{proposition: unique solution of control equation} implies that the second term vanishes as well.
	
	We obtain
	\begin{equation*}
		g(\bd^\star  + \be(r, \bb')) - g(\bd^\star  + \be^\star ) - \langle \nabla  g(\bd^\star  + \be^\star ), \be(r, \bb') - \be^\star  \rangle = O((\beta(r) + \|\be^\star \|^2)  \|\be(r, \bb') - \be^\star \|)\,.
	\end{equation*}
	Since $\|\be^\star \| = O(\|\bb' - \bb\|/r) = o(1)$, we can choose a compact set $K \subseteq \dom g$ such that $\bd^\star  + \be^\star $ remains in $K$.
	
	Applying Lemma~\ref{lem:new_norm_bound} and Lemma~\ref{lemma: g convex on affine} with $\mathbb{A}=\{\bm{x}\in\mathbb{R}^m|\bm{Ax}=\bm{b}'-\bm{b}\}$, we conclude as in the proof of Theorem~\ref{theorem: unique solution general penalty}.
\end{proof}

\subsubsection{Convergence law for random penalized program}
The convergence law in equation~\eqref{converge bn} for the random variable $\bm{b}_n$ indicates that, with high probability, $\|\bm{b}_n-\bm{b}\|$ falls within the interval $\left(r_n\beta(r_n),r_n\right)$. We leverage our  understanding of the convergence trajectory for the non-random perturbed problem, characterized in Theorem~\ref{thm: trajectory with noise}, to demonstrate the convergence law of $\bm{x}(r_n,\bm{b}_n)$.

\begin{proof}[Proof of Theorem~\ref{theorem: limit law}]

Since $r_n\beta(r_n)\ll q_n^{-1}\ll r_n$, we can choose sequences $u_n$ and $l_n$ such that $r_n\beta(r_n)\ll l_n \ll q_n^{-1} \ll u_n \ll \sqrt{r_n q_n^{-1}}$. Therefore, we have 
\begin{equation}
    P(\|\bm{b}_n-\bm{b}\|\in [l_n,u_n])\rightarrow 1.
\end{equation}
    
    \BN Let \EN $\bm{\eta}(r_n,\bm{b}_n-\bm{b})\BN :\EN={\bm{x}}(r_n,\bm{b}_n)-r_n\bm{d}^\star -\bm{x}^\star -\bm{M}^\star(\bm{b}_n-\bm{b})$. When $\|\bm{b}_n-\bm{b}\|\in [l_n,u_n]$, Theorem~\ref{thm: trajectory with noise} indicates that $$q_n\bm{\eta}(r_n,\bm{b}_n-\bm{b})\mathbbm{1}_{ \|\bm{b}_n-\bm{b}\|\in [l_n,u_n]}=o_p(1).$$
    Therefore, as a result of the Slusky's theorem and \BN Theorem \EN~\ref{thm: trajectory with noise}, 
    \begin{align}
    \begin{split}
    q_n({\bm{x}}(r_n,\bm{b}_n)-r_n\bm{d}^\star -\bm{x}^\star )=&q_n\bm{M}^\star(\bm{b}_n-\bm{b})+q_n\bm{\eta}(r_n,\bm{b}_n-\bm{b})\mathbbm{1}_{ \|\bm{b}_n-\bm{b}\|\in [l_n,u_n]}\\
    &+ q_n\bm{\eta}(r_n,\bm{b}_n-\bm{b})\mathbbm{1}_{\{l_n\geq \|\bm{b}_n-\bm{b}\|\}\cup\{\|\bm{b}_n-\bm{b}\|\geq u_n\}}\\
    =&q_n\bm{M}^\star(\bm{b}_n-\bm{b})+o_p(1)\overset{D}\to\bm{M}^\star(\mathbb{G}) 
    \end{split}
\end{align}
\end{proof}
Now we establish the asymptotic law of the debiased estimator. Our proof demonstrates that by utilizing the two solutions under different penalty strengths, we can construct an effective estimator for the bias vector $\bm{d}^\star $. 
 \begin{proof}[Proof of Theorem~\ref{theorem: debiasing}]
  Define an estimator for $\bm{d}^\star $ as
    \begin{equation}
        \hat{\bm{d}}_n=\frac{2}{r_n}({\bm{x}}(r_n,\bm{b}_n)-{\bm{x}}(\frac{r_n}{2},\bm{b}_n)).
    \end{equation}

     Consider the event $H_n\coloneqq\{l_n\leq \|\bm{b}_n-\bm{b}\|\leq u_n\}$, where $l_n$ and $u_n$ are sequences satisfying $r_n\beta(r_n)\ll l_n \ll q_n^{-1} \ll u_n \ll  \sqrt{r_n q_n^{-1}}$. We have $\lim_{n\rightarrow\infty}\mathbb{P}(H_n)=1$.

    On the event $H_n$, we have $$\|r_n(\hat{\bm{d}}_n-\bm{d}^\star )\|\leq2\|\eta(r_n,\bm{b}_n-\bm{b})\|+2\|\eta(\frac{r_n}{2},\bm{b}_n-\bm{b})\|,$$
    where $\eta(r_n,\bm{b}_n)={\bm{x}}(r_n,\bm{b}_n)-(\bm{x}^\star +r_n\bm{d}^\star +\bm{M}^\star(\bm{b}_n-\bm{b}))
    \lesssim O_p(\max(\frac{\|\bm{b}_n-\bm{b}\|^2}{r_n},r_n\beta(r_n))).$

Therefore, we have $$\|\hat{\bm{d}}_n-\bm{d}^\star \|\leq O_p(\max(\frac{\|\bm{b}_n-\bm{b}\|^2}{r_n^2},\beta(r_n)))=o_p(1).$$
Therefore, $\hat{\bm{d}}_n$ is an asymptotically consistent estimator for $\bm{d}^\star $.
Furthermore,
$$q_n r_n\|\hat{\bm{d}}_n-\bm{d}^\star \|\leq O_p(\max(\frac{q_n\|\bm{b}_n-\bm{b}\|^2}{r_n},q_n r_n\beta(r_n)))=o_p(1).$$
Combining this result with Theorem~\ref{theorem: limit law}, we have 
\begin{equation}
        q_n(\hat{\bm{x}}_n-\bm{x}^\star )\overset{D}\to\bm{M}^\star(\mathbb{G}).
    \end{equation}   
\end{proof}

The following corollary establishes a remarkable convergence property for the objective function value. Under specific conditions, the  debiased estimator $\hat{\bm{x}}_n$ achieves an accelerated convergence to the optimal value of the linear program~\eqref{primal}.
 
\begin{corollary}[fast convergence of optimal value]
\label{corollary: fast convergence of cost}
     Let $\bm{x}^\star (\bm{b})$ be the solution for the linear program~\eqref{primal} with the parameter $\bm{b}$.
    If $\bm{c}\geq 0$, $\langle \bm{c},\bm{x}^\star  \rangle =0$, and there exists $\epsilon>0$, $\text{supp}(\bm{x}^\star (\bm{b}+\epsilon \bm{A}\bm{\rho}))=\text{supp}(\bm{x}^\star (\bm{b}))$ for $\bm{\rho}= \ln(\bm{c})\cdot\bm{\mathbbm{1}}_{\{\bm{x}^\star =0\}}$ and for $\bm{\rho}= \bm{\mathbbm{1}}_{\{\bm{x}^\star =0\}}$, with the exponential penalty function,
    \begin{equation}
        q_n(\langle\bm{c},\hat{\bm{x}}_n\rangle-\langle\bm{c},{\bm{x}^\star }\rangle)\overset{D}\to 0.
    \end{equation} 
\end{corollary}
\begin{proof}
    If is enough to show under given conditions, $\langle\bm{c},\bm{M}^\star(\mathbb{G})\rangle =0$.

    The condition $\text{supp}(\bm{x}^\star (\bm{b}+\epsilon \bm{A}\bm{\rho}))=\text{supp}(\bm{x}^\star (\bm{b}))$ indicates that the  equations $\{\bm{Ax}=0,\ \bm{x}\mathbbm{1}_{\bm{x}^\star =0}=-\bm{\rho}\}$ has a solution $\bm{x}_{\bm{\rho}}=\bm{-\rho+\frac{\bm{x}^\star (\bm{b}+\epsilon \bm{A}\bm{\rho})-\bm{x}^\star (\bm{b})}{\epsilon}}.$

    Therefore, when taking $\bm{\rho}_0= \ln(\bm{c})\cdot\bm{\mathbbm{1}}_{\{\bm{x}^\star =0\}}$, 
    we have $\bm{x}_{\bm{\rho}_0}$ satisfy
    \begin{equation}
    \label{equ corrolory optimal}
        \bm{c}=\exp(-\bm{x}_{\bm{\rho}_0})\bm{\mathbbm{1}}_{\{\bm{x}^\star =0\}},\quad \bm{A\bm{x}_{\bm{\rho}_0}}=0.
    \end{equation}
    This implies that $\bm{x}_{\bm{\rho}_0}$ satisfies the optimality condition of~\eqref{equ for 1st order} with the exponential penalty function hence $\bm{d}^\star =\bm{x}_{\bm{\rho}_0}$.
    
    Also, when taking $\bm{\rho}_1= \bm{\mathbbm{1}}_{\{\bm{x}^\star =0\}}$. By the definition of  $\bm{M}^\star$, we have  $$\langle\bm{x}_{\bm{\rho}_1},\text{diag}(p''(-\bm{d}^\star )\bm{\mathbbm{1}}_{\{\bm{x}^\star =0\}})\bm{M}^\star(\mathbb{G})\rangle =0.$$
    Since $\bm{x}_{\bm{\rho}_1}\bm{\mathbbm{1}}_{\{\bm{x}^\star =0\}}=\bm{\mathbbm{1}}_{\{\bm{x}^\star =0\}}$ and $p''(-\bm{d}^\star )=\exp(-\bm{x}_{\bm{\rho}_0})$,combined with~\eqref{equ corrolory optimal}, this further implies that 
    $$\langle \bm{\alpha},\text{diag}(p''(-\bm{d}^\star )\bm{\mathbbm{1}}_{\{\bm{x}^\star =0\}})\bm{M}^\star(\mathbb{G})\rangle=\langle p''(-\bm{d}^\star )\bm{\mathbbm{1}}_{\{\bm{x}^\star =0\}},\bm{M}^\star(\mathbb{G})\rangle=\langle\bm{c},\bm{M}^\star(\mathbb{G})\rangle =0.$$
\end{proof}

Corollary~\ref{corollary: OT fast convergence} applies the results of Corollary~\ref{corollary: fast convergence of cost} to the special case of  optimal transport problems where the source and target distributions are identical. 

 \begin{proof}[Proof of Corollary~\ref{corollary: OT fast convergence}]
 The proof consists of verifying that all conditions specified in Corollary~\ref{corollary: fast convergence of cost} are satisfied.

Since $\bm{t}=\bm{s}$ and $c_{i,i}=0$, the optimal transport plan $\bm{\pi}^\star $ only has support on the diagonal entries.   For the symmetric cost matrix,  $\bm{\rho}$ is a symmetric matrix in both scenarios. Therefore, for small enough $\epsilon$, the transportation plan $\bm{\pi}^\star (\bm{t}+\epsilon \sum_{j=1}^N \rho_{i,j},\bm{s}+\epsilon \sum_{i=1}^N \rho_{i,j})$ also has support on the diagonal entries.  \end{proof}

\subsubsection{Bootstrap consistency}
Finally, we demonstrate that the naive bootstrap estimator $\tilde{\bm{x}}_n$, generated using the bootstrap copy of $\bm{b}_n$, is consistent for conducting inferences on $\bm{x}^\star $.

\begin{proof}[Proof of Theorem~\ref{theorem: bootstrap}]
    The proof mainly follows the proof of~\citep[Theorem 3.9.11]{VaaWel96}. 
 Since $\tilde{\bm{b}}_n$ is a bootstrap copy of $\bm{b}_n$ through resampling and the function $h(\bm{M}^\star(\cdot))$ is bounded Lipschitz functions for every $h\in BL_1(\mathbb{R}^k)$,  we have
    \begin{equation}
        \sup_{h \in BL_1(\mathbb{R}^k)} \left| \mathbb{E}\left[h\left(\sqrt{n}\bm{M}^\star(\tilde{\bm{b}}_n - \bm{b}_n)\right) | Z_1, \ldots, Z_n\right] - \mathbb{E}\left[h\left(\sqrt{n}\bm{M}^\star(\bm{b}_n - \bm{b})\right)\right] \right|\overset{P}\to 0.
    \end{equation}

    Since the random vector $\tilde{\bm{b}}_n$ has the same convergence law as $\bm{b}_n$, Theorem~\ref{theorem: limit law} and~\ref{theorem: debiasing} also holds for ${\bm{x}}(r_n,\tilde{\bm{b}}_n)$ and for $\tilde{\bm{x}}_n$. Therefore, we have
    \begin{equation}
        \sqrt{n}(\tilde{\bm{x}}_n - \bm{x}^\star )=\bm{M}^\star(\tilde{\bm{b}}_n-\bm{b})+o_p(1),
    \end{equation}
    and
    \begin{equation}
        \sqrt{n}(\hat{\bm{x}}_n - \bm{x}^\star )=\bm{M}^\star(\bm{b}_n-\bm{b})+o_p(1).
    \end{equation}
    Thus, by subtracting the above two equations we have
    \begin{equation}
        \sqrt{n}(\tilde{\bm{x}}_n-\hat{\bm{x}}_n -\sqrt{n}\bm{M}^\star(\tilde{\bm{b}}_n - \bm{b}_n))\overset{P}\to 0.
    \end{equation}
    Furthermore, for every $\epsilon>0$
    \begin{align*}
         &\sup_{h \in BL_1(\mathbb{R}^k)} \left| \mathbb{E}\left[h\left(\sqrt{n}(\tilde{\bm{x}}_n - \hat{\bm{x}}_n)\right) | Z_1, \ldots, Z_n\right] - \mathbb{E}\left[h\left(\sqrt{n}\bm{M}^\star(\tilde{\bm{b}}_n - \bm{b}_n)\right) | Z_1, \ldots, Z_n\right]\right|\\
         &\leq \epsilon+2\mathbb{P}\left(  \|\sqrt{n}(\tilde{\bm{x}}_n-\hat{\bm{x}}_n -\sqrt{n}\bm{M}^\star(\tilde{\bm{b}}_n - \bm{b}_n))\|\geq\epsilon| Z_1, \ldots, Z_n\right).
    \end{align*}
    Combining with the convergence law for $ \hat{\bm{x}}_n$ given in Theorem~\ref{theorem: debiasing}, we get the desired statement~\eqref{equation: theorem bootstrap}.
    
\end{proof}

\bibliographystyle{abbrvnat}
\bibliography{ref}

\end{document}